\documentclass{aims}
\usepackage{booktabs}
\usepackage{enumerate}
\usepackage{undertilde}
\usepackage{array}
\newcommand{\PreserveBackslash}[1]{\let\temp=\\#1\let\\=\temp}
\newcolumntype{C}[1]{>{\PreserveBackslash\centering}p{#1}}
\newcolumntype{R}[1]{>{\PreserveBackslash\raggedleft}p{#1}}
\newcolumntype{L}[1]{>{\PreserveBackslash\raggedright}p{#1}}
\usepackage{amsmath}
 \usepackage{paralist}
  \usepackage{graphics} %% add this and next lines if pictures should be in esp format
  \usepackage{epsfig} %For pictures: screened artwork should be set up with an 85 or 100 line screen
\usepackage{graphicx}
\usepackage{epstopdf}%This is to transfer .eps figure to .pdf figure; please compile your paper using PDFLeTex or PDFTeXify.
 \usepackage[colorlinks=true]{hyperref}
 \usepackage{undertilde}
\hypersetup{urlcolor=blue, citecolor=red}

\def\currentvolume{X}
 \def\currentissue{X}
  \def\currentyear{200X}
   \def\currentmonth{XX}
       \def\ppages{X--XX}
     \def\DOI{}

\newtheorem{theorem}{Theorem}[section]

\newtheorem{lemma}[theorem]{Lemma}

\theoremstyle{definition}
\newtheorem{definition}[theorem]{Definition}

\title[Fractional Predator-Prey Reaction-Diffusion Model]{Positivity and Boundedness Preserving Schemes
 for Space-Time Fractional Predator-Prey Reaction-Diffusion Model}

\author{Yanyan Yu, Weihua Deng, and Yujiang Wu}

\subjclass{Primary: 65M06, 26A33; Secondary: 45M20.}
 \keywords{semi-implicit scheme, fractional predator-prey model, fractional centered difference, WSGD scheme, convergence.}

 \email{yuyanyan2009@lzu.edu.cn}
 \email{dengwh@lzu.edu.cn}
 \email{myjaw@lzu.edu.cn}

\begin{document}

\maketitle

% Enter the first author's name and address:
% \centerline{\scshape Yanyan Yu, Weihua Deng, and Yujiang Wu}
%\medskip
{%\footnotesize
% please put the address of the first author
 \centerline{School of Mathematics and Statistics, Lanzhou University, }
%   \centerline{ }
   \centerline{ Lanzhou, 730000, People's Republic of China }
} % Do not forget to end the {\footnotesize by the sign }

%\medskip
%
%\centerline{\scshape Weihua deng and Yujiang wu}
%\medskip
%{\footnotesize
% % please put the address of the second  and third author
% \centerline{ School of Mathematics and Statistics }
%   \centerline{ Lanzhou University}
%   \centerline{ Lanzhou, 730000, People's Republic of China}
%}

\bigskip

% The name of the associate editor will be entered by an editorial staff
% "Communicated by the associate editor name" is not needed for special issue.
% \centerline{(Communicated by the associate editor name)}

%The abstract of your paper
\begin{abstract}
The semi-implicit schemes for the nonlinear predator-prey reaction-diffusion model with the space-time fractional derivatives are discussed, where the space fractional derivative is discretized by the fractional centered difference and WSGD scheme. The stability and convergence of the semi-implicit schemes are analyzed in the $L_\infty$ norm.  We theoretically prove that the numerical schemes are stable and convergent without the restriction on the ratio of space and time stepsizes and numerically further confirm that the schemes have first order convergence in time and second order convergence in space. Then we discuss the positivity and boundedness properties of the analytical solutions of the discussed model, and show that the numerical solutions preserve the positivity and boundedness. The numerical example is also presented.
\end{abstract}

\section{Introduction}

The predator-prey model, also known as the Lotka-Volterra equation, is a pair of first-order nonlinear ordinary differential equations frequently used to describe the dynamics of biological system in which two species interact, one as the predator and the other as prey \cite{Brauer00}; and the unknown variables usually denote certain measure of total population.  To interpret the unknown variables as spacial densities so as to allow the population size to vary throughout the considered region, Conway and Smoller introduce diffusion terms to the predator-prey model which allows diffusing as well as interacting each other \cite{Conway77}; and then the predator-prey reaction-diffusion model is obtained. With the development of the predator-prey model, it seems nature to introduce diffusion terms to the corresponding model, e.g., the Michaelis-Menten-Holling predator-prey reaction-diffusion model \cite{Cavani1994,aly2011}; and more general other models \cite{Bartumeus2001,Pang2003,Wang2003,ana2011}.

Anomalous diffusion, including subdiffusion and superdiffuion, is also a diffusion process, but its mean squared displacement (MSD) is nonlinear with respect to time $t$, in contract to the classical diffusion process, in which the MSD is linear \cite{Metzler00}; nowadays, it is widely recognized that the anomalous diffusion is ubiquitous, e.g.,  diffusion through porous media, protein diffusion within cells, and also being found in many other biological systems. So it seems reasonable/nature to introduce the anomalous diffusion to the predator-prey model. The subdiffusion is introduced to the Michaelis-Menten-Holling predator-prey model in \cite{yu2013} to get a new model; it is proved that the solution of the model is positive and bounded; and the numerical schemes preserving the positivity and boundedness are detailedly discussed. Here we introduce both the subdiffusion and superdiffusion to the Michaelis-Menten-Holling predator-prey model, and the system can be written as

% are increasingly favored by some animal ecologists recently.
%%They can better describe    predator-prey interactions.
%The predator and prey within an ecosystem are interdependent.
%The change  of one population
%or the environment they live affects the other organism within the
%ecosystem.
% This is  feature of the relationship
%between predator and prey populations.
%The coupled nonlinear predator-prey equations are formatted.
% Cavani and
%Farkas  introduced diffusion term to the
%Michaelis-Menten-Holling predator-prey model in \cite{Cavani1994}.  And the more general
%models are considered in \cite{Bartumeus2001,Pang2003,Wang2003,ana2011}.
%In this paper, we discuss a more suitable general  predator-prey model, the fractional predator-prey model.
% % which is based on the so called ratio-dependent theory.
%
%
%
% In recent years, researchers found that fractional partial differential equations can explain some   physical and ecological phenomena
% better than integral partial differential equations.
% There exist some results about the boundary and positivity of analytical  and numerical solutions of  the coupled fractional reaction sub-diffusion equations\cite{yu2013}.
%%We further discuss the fractional predator-prey model, as  an  important  example.
% Here
%we further discuss the Michaelis-Menten-Holling predator-prey model\cite{aly2011}:

\begin{equation}\label{Eq1.1}
\begin{array}{llll}
\displaystyle\frac{\partial^\alpha N}{\partial t^\alpha} &= & D_1\displaystyle\frac{\partial^\beta N}{\partial |x|^\beta}+N\left(1-N-\frac{\varrho P}{P+N}\right), & x \in (l,r),\,\,t>0,\\
\\
\displaystyle\frac{\partial^\alpha P}{\partial t^\alpha} &= &D_2
\displaystyle\frac{\partial^\beta P}{\partial |x|^\beta}+\sigma P
\left(-\frac{\gamma+ \kappa\delta P}{1+\kappa P}+\frac{N}{P+N}\right),
& x \in (l,r),\,\,t>0,
\end{array}
\end{equation}
with the Caputo derivative in time and Riesz space fractional derivetive, where  $\varrho$, $\sigma$  and $\kappa$  are positive real numbers and $N$ and $P$ denote the population densities of prey and predator  respectively. Based on the practical applications, we are interested in the solutions of (\ref{Eq1.1}) with the nonnegative initial conditions
 \begin{equation}\label{Eq1.1I}
 N(x,0)=g_1(x) \ge 0,~~~ P(x,0)=g_2(x) \ge 0, ~~~ x \in (l,r),
  \end{equation}
and the homogeneous Neumann boundary conditions
 \begin{equation}\label{Eq1.2}
(\partial {N(x,t)}/\partial x) |_{x=l \rm{~and~} r, \rm{~respectively}}=(\partial {P(x,t)}/\partial x) |_{x=l \rm{~and~} r, \rm{~respectively}}=0.
 \end{equation}
 The positive constants $\gamma$ and $ \delta$ in the coupled equations denote the minimal mortality and the limiting mortality of the predator, respectively. Throughout the paper, we assume that $\gamma$ satisfies the natural condition  $0<\gamma \le \delta$ and consider the case of the diffusion constants $D_i>0,\, i=1,2$.
%Next we give that the analytical and numerical solutions of (\ref{Eq1.1}) and
%(\ref{Eq1.2})  is positive and bounded.

The numerical methods for solving fractional partial differential equations are developing fast; most of them focus on fractional diffusion equations, including the space fractional diffusion equation \cite{Tadjeran2006,Yang2010} and the time fractional diffusion equation \cite{Chen2007,Cui2009, Gao2011,Lin2007,Yuste2005, Deng2007JCP}. Because of the stability issue, the space fractional derivative is usually approximated by the shifted Gr\"{u}nwald-Letnikov definition with the finite stepsize; and the truncation error is first-order. Recently, the second-order discretizations for space fractional derivative appear: based on the so-called ``fractional centered difference", Ortigueira gets the second-order approximation for the Riesz fractional derivative \cite{Ortigueira2006}, and its applications can be seen in \cite{Celik2012,Wang2013}; Tian et al obtain the so-called WSGD second order approximation for both the left and right Riemann-Liouville derivatives \cite{Tian2012}, and its compact version has third-order accuracy \cite{Zhou2013}.

This paper first proves that the analytical solutions of the space-time fractional predator-prey reaction-diffusion model (\ref{Eq1.1})-(\ref{Eq1.2}) are positive and bounded; and then designs the numerical schemes to solve it. The space fractional derivative is discretized by the fractional centered difference \cite{Ortigueira2006} and the WSGD operators \cite{Tian2012}, respectively. We prove that both the two obtained schemes have second-order accuracy in space and preserve the positivity and boundedness of the analytical solutions. In particular, the stability of the numerical scheme without the restriction on the ratio of the space and time stepsizes is also strictly proved. And the numerical example is provided to confirm the theoretical results.

%In \cite{Ortigueira2006}, Ortigueira defined a new approach as ``fractional centered derivative" and  gives the Riesz fractional derivative  represented by the fractional centered derivative. Celik and Duman\cite{Celik2012} pointed out that the numerical method to approximate the Riesz fractional derivative has second order accuracy in space and applied the  Crank-Nicolson method to a fractional equation to obtain a unconditionally stable and convergent scheme. Wang et al. \cite{Wang2013} used this scheme to solve a coupled nonlinear Schr\"{o}dinger equations and got  better results.

The outline of this paper is as follows. In Section 2, we present two finite difference schemes for the space-time fractional predator-prey reaction-diffusion model. The detailed discussions on the stability and convergence with first-order in time and second-order in space are given in Section 3. In Section 4, we prove that the analytical solutions of the discussed model are positive and bounded; and the provided two schemes preserve the positivity and boundedness. The numerical experiments to confirm the convergent orders and the positivity and boundedness preserving are performed in Section 5. And we conclude the paper with some discussions in the last section.

%
% And then discuss the stability and
%convergence of the given scheme, and get that the
%temporal approximation order is $1$
%and the spatial order  is
%$2$ in the next part.  In Section 4, we   certify the positivity and boundedness of
%the solution of the  space-time fractional predator-prey model, and prove
%that the proposed numerical schemes preserve the   positivity and
%boundedness. In   Section 5, we perform some numerical experiments to
%confirm the  orders and positivity and boundedness
%preserving.

%-------------------------------------------------------------------------------------------------------------------------------------------------------
\section{Numerical schemes for the space-time fractional predator-prey reaction-diffusion model}

As mentioned in the introduction section, the model we discuss is as follows:
\begin{equation}\label{2.1}
\begin{array}{llll}
\displaystyle\frac{\partial^\alpha N}{\partial t^\alpha} &= & D_1\displaystyle\frac{\partial^\beta N}{\partial |x|^\beta}+N\left(1-N-\frac{\varrho P}{P+N}\right), & x \in (l,r),\,\,0<t\leq T,\\
\\
\displaystyle\frac{\partial^\alpha P}{\partial t^\alpha} &= &D_2
\displaystyle\frac{\partial^\beta P}{\partial |x|^\beta}+\sigma P
\left(-\frac{\gamma+\kappa\delta   P}{1+\kappa P}+\frac{N}{P+N}\right),
& x \in (l,r),\,\,0<t\leq T,
\end{array}
\end{equation}
with $0<\alpha<1$, $1<\beta<2$, and use the following
initial conditions
\begin{equation}\label{2.2}
 N(x,0)=g_1(x),\, P(x,0)=g_2(x),\qquad x \in [l,r]
\end{equation}
and the homogeneous boundary conditions
%\begin{equation}
%u(0,t)=u(1,t)=0,\quad 0< t\leq T,
%\end{equation}
%or
\begin{equation} \label{2.3}
  \frac{ \partial {N(x,t)}}{\partial x} |_{x=l }=\frac{ \partial {N(x,t)}}{\partial x} |_{x=r }=  \frac{ \partial {P(x,t)}}{\partial x} |_{x=l }=\frac{ \partial {P(x,t)}}{\partial x} |_{x= r }=0,\quad 0< t\leq T.\\
\\
\end{equation}
In (\ref{2.1}), the time fractional operator $\frac{\partial^\alpha u(x,t)}{\partial t^\alpha}$ denotes the
 Caputo fractional derivative
\begin{equation*}
\frac{\partial^\alpha u(x,t)}{\partial
 x ^\alpha}=\frac{1}{ \Gamma(1-\alpha)} \int^{t}_{0}  \frac{\partial u(x,s)}{\partial s}\frac{1}{(t-s)^\alpha}ds ,
\end{equation*}
 and $\frac{\partial^\beta u}{\partial |x|^\beta}$ denotes the Riesz fractional derivative
%
%here, we consider the sub-diffusion, namely $0<\alpha<1$. Let us calculate the derivative $_0 D^{\alpha}_t u(x,t)$ in the form of Caputo fractional derivatives\cite{I1998}
\begin{equation}
\frac{\partial^\beta u(x,t)}{\partial
|x|^\beta}=-\frac{1}{2cos(\beta \pi /2)\Gamma(2-\beta)}\frac{d^2}{dx^2}\int^{r}_{l} |x-\xi|^{1-\beta}u(\xi,t)d\xi .
\end{equation}
%We apply the fractional centered difference to approximate the  Riesz fractional derivative:
%The fractional centered difference for  Riesz fractional derivative is defined by

For ease of presentation, we uniformly divide the spacial domain
$[l,r]$ into $M$ subintervals with stepsize $h$ ($=(r-l)/M$) and the time
domain $[0,T]$ into $N$ subintervals with steplength $\tau$ ($=T/N$).
Let $x_{i}=l+i h\,(i=0,1,\cdots,M)$, $t_{k}=k\tau \,( k=0,1,\cdots,N)$, and denote the grid function as $u^{k}_{i}=u(x_i,t_k)$.
%\, 0\leq i\leq M,\,0\leq k\leq N}$.
We use the discrete scheme of the Caputo derivative in the following form \cite{Lin2007}
\begin{equation}\label{2.6}
D^{\alpha}_{\tau}u_i^{k}=\frac{\tau^{-\alpha}}{\Gamma(2-\alpha)}\big[u_i^{k}-\sum^{k-1}_{n=1}(b_{k-n-1}-b_{k-n})u_i^{n}-b_{k-1}u_i^{0}\big],
\end{equation}
where $b_{n}=(n+1)^{1-\alpha}-n^{1-\alpha}>0$, and $1=b_0>b_1>\cdots>b_{k-1} >
(1-\alpha)k^{-\alpha}$.  %There exists the following
%error estimate between $(\partial^\alpha u(x_i,t)/\partial
%t^\alpha)|_{t=t_n}$ and $D^{\alpha}_{\tau}u_i^{n}$:
Note that  the above discrete scheme has  $2-\alpha$  order accuracy in time, i.e.,
%\begin{lemma}
%For $0<\alpha<1$ and $u(x_i, t)\in \mathcal{C}^2[0,T]$, it holds that
\begin{equation}\label{lemma2.3}
\frac{\partial^\alpha u(x_i,t)}{\partial
t^\alpha}\big|_{t=t_k}=D^{\alpha}_{\tau}u_i^k+\mathcal{O}(\tau^{2-\alpha}).
\end{equation}
%\end{lemma}
For the Riesz space fractional derivative, we introduce two second order finite difference approximation schemes (fractional centered difference scheme and second order WSGD scheme) in the next parts.

\subsection{Fractional centered difference scheme}
In this part, we use the fractional centered difference  \cite{Ortigueira2006} to discretize the spatial factional derivative.
 The Riesz fractional derivative can be redefined as
\begin{equation}\label{fcd}
\frac{\partial^\beta u(x,t)}{\partial
|x|^\beta}=\lim_{h\rightarrow0} \left(-\frac{1}{h^{\beta}}\sum_{j=[(l-r+x)/h]}^{[x/h]}\frac{(-1)^j\Gamma(\beta+1)}{\Gamma(\frac{\beta}{2}-j+1)\Gamma(\frac{\beta}{2}+j+1)}u(x-jh,t)\right).
\end{equation}
Denoting the weights as
 $$g_j=\frac{(-1)^j\Gamma(\beta+1)}{\Gamma(\frac{\beta}{2}-j+1)\Gamma(\frac{\beta}{2}+j+1)},$$
we get the fractional centered difference approximation for the Riesz fractional derivative
%the discrete scheme of the Riesz fractional derivative is defined as
%represents the Riesz fractional derivative for $1<\beta<2$.
%We introduce the following discrete representation for Riesz fractional derivative
%
%which is based on the fractional centered difference.
% and define as fractional centered difference scheme(FCDS).
%Applying  (\ref{fcd}) to   the  Riesz fractional derivative as a discretization,   we get the following lemma.
\begin{equation} \label{2.5}
\delta_x^\beta u_i^k=-\frac{1}{h^\beta}\sum_{j=-M+i}^ig_ju_{i-j}^k,
\end{equation}
where  the weights $g_j$  satisfy the following properties:
%\begin{lemma} ~\cite{Ortigueira2006} \label{lemma2.1}
%Let $g_j=\frac{(-1)^j\Gamma(\beta+1)}{\Gamma(\frac{\beta}{2}-j+1)\Gamma(\frac{\beta}{2}+j+1)}$ for any $j=\pm1,\pm2,\cdots$, it hold that
\begin{enumerate}
\item $g_0\geq0$,
\item $g_{-j}=g_{j}\leq0$,
\item $g_{j+1}=\left(1-\frac{\beta+1}{\beta/2+j+1}\right)g_j$, and $|g_{j+1}|<|g_j|$,
\item $\sum\limits_{j=-\infty}^{\infty}g_j=0$, and $\sum\limits_{j=-M+i \atop j\neq0}^{i }|g_j| < g_0$.
\end{enumerate}
%\end{lemma}
From \cite{Celik2012}, we know that the accuracy of fractional centered difference approximation is as,
\begin{lemma} \label{lemma2.2}
If $u(x,t_i)\in \mathcal{C}^5[l,r]$, then we have
$$
\frac{\partial^\beta u(x,t_i)}{\partial
|x|^\beta}|_{x=x_i}=\delta_x^\beta u_i^k+\mathcal{O}(h^2),
$$
with   $1<\beta<2$.
\end{lemma}
If the fractional centered difference approximation is substituted into the system of reaction-diffusion equation (\ref{2.1}),
then the resulting semi-implicit finite difference scheme is
%Considering the equations (\ref{2.1}) at the point $(x_i,t_n)$, we have the discrete scheme
\begin{equation}\label{2.8}
\begin{array}{llll}
\displaystyle D^{\alpha}_{\tau}N_i^{k+1}&= &D_1 \displaystyle\delta_x^\beta N_i^{k+1}+N_i^k\left(1-N_i^k-\frac{\varrho P_i^k}{P_i^k+N_i^k}\right), \\
\\
\displaystyle D^{\alpha}_{\tau}P_i^{k+1} &= &D_2
\displaystyle\delta_x^\beta P_i^{k+1}+\sigma P_i^k
\left(-\frac{\gamma+\kappa\delta   P_i^k}{1+\kappa P_i^k}+\frac{N_i^k}{P_i^k+N_i^k}\right),
\end{array}
\end{equation}
with initial conditions
%The numerical approximation of initial  conditions (\ref{2.2})-(\ref{2.3}) are given by
\begin{equation}\label{2.9}
N_i^0=g_1(x_i),\, P_i^0=g_2(x_i).
\end{equation}
For guaranteeing the second-order accuracy in space, we give the three point interpolation scheme for the Neumann boundary conditions:
\begin{equation}\label{2.10}
\begin{array}{llll}
 \quad\quad\quad  N_2^k-4N_1^k+3N_0^k&= 0, \quad  \displaystyle& P_2^k-4P_1^k+3P_0^k=0,\quad 0\leq k \leq N,\\\\
  N_{M-2}^k-4N_{M-1}^k+3N_M^k&= 0, \quad \displaystyle& P_{M-2}^k-4P_{M-1}^k+3P_M^k=0,\quad 0\leq k \leq N.
\end{array}
\end{equation}
Eq. (\ref{2.8})-(\ref{2.10}) may be rearranged and written as matrix form and solved to get the numerical solution at the time step $t_{k+1}$:
  $$
\begin{aligned}
 (I+A_1)N^{k+1}&=\sum_{j=1}^{k}(b_j-b_{j+1})N^{k-j}+b_kN^0+\mu  N^k\left(1-N^k-\frac{\varrho P^k}{P^k+N^k}\right),\\
 (I+A_2)P^{k+1}&=\sum_{j=1}^{k}(b_j-b_{j+1})P^{k-j}+b_kP^0+\mu \sigma P^k
\left(-\frac{\gamma+\kappa\delta   P^k}{1+\kappa P^k}+\frac{N^k}{P^k+N^k}\right),
\end{aligned}
  $$
  where $N^{k}=(N_1^{k},N_2^{k},\cdots,N_{M-1}^{k})$,  $P^{k}=(P_1^{k},P_2^{k},\cdots,P_{M-1}^{k})$ and $A_i$ is a $(M-1)\times(M-1)$
  square matrix
   \begin{equation}\label{2.11}       %开始数学环境
A_i=\frac{\mu D_i}{h^\beta}\left(                 %左括号
  \begin{array}{ccccc}   %该矩阵一共3列，每一列都居中放置
     g_0 +\frac{4}{3} g_{1}   & g_{-1} - \frac{1}{3} g_{1} & g_{-2} &  \cdots    &g_{-M+2} +\frac{4}{3} g_{-M+1}  \\  %第一行元素
    g_1  +\frac{4}{3} g_{2}    & g_0 -\frac{1}{3} g_{2}    &  g_{-1}  &           & g_{-M+3}+\frac{4}{3} g_{-M+2}   \\  %第二行元素
    \vdots      &        &          &   \ddots           &   \\
    g_{M-2} +\frac{4}{3} g_{M-1} &  g_{M-3}-\frac{1}{3} g_{M-1} & g_{M-4}&            & g_0 +\frac{4}{3} g_{-1}\\
  \end{array}
\right).                 %右括号
\end{equation}

\subsection{Second order WSGD scheme}
 The other equivalent definition of the Riesz fractional derivative $\frac{\partial^\beta u(x,t)}{\partial|x|^\beta }$ is
\begin{equation}\label{sch21}
\frac{\partial^\beta u(x,t)}{\partial|x|^\beta }=\frac{-1}{2cos(\beta\pi/2)}(_{l}D_x^\beta u(x,t)+_xD_{r}^\beta u(x,t)),
\end{equation}
where $_{l}D_x^\beta u(x,t)$ and $_xD_{r}^\beta u(x,t)$ denote the left and right Riemann-Liouville fractional derivatives of the function $u(x,t)$, respectively.
%According to the reference
According to the discussions of \cite{Tian2012},  we present the discrete approximations of the left and right Riemann-Liouville fractional derivatives as follows:
\begin{equation}\label{sch22}
\begin{array}{c}
\displaystyle_lD^\beta_xu(x_i)=\frac{1}{h^\beta}\sum_{j=0}^{i+1}w_j u(x_{i-j+1})+\mathcal{O}(h^2),\\\\
\displaystyle_xD^\beta_ru(x_i)=\frac{1}{h^\beta}\sum_{j=0}^{M-i+1}w_j u(x_{i+j-1})+\mathcal{O}(h^2),
\end{array}
\end{equation}
where $w_0=\alpha/2$, $w_j=\Gamma(j-\beta-1)/(\Gamma(-\beta)\Gamma(j))\cdot (1-(\beta/2)\cdot(\beta+1)/j)$. % with $(p,q)=(1,0)$ or $(1,-1)$.
 After substituting (\ref{sch22}) into (\ref{sch21}) and merging the same terms, we get the discrete approximation of the Riesz fractional derivative
 $$
 \delta_xu_i^k=\frac{-1}{h^\beta}\sum_{j=-M+i}^{i}\theta_ju_{i-j}^k,
 $$
where the weights $\theta_j$ satisfy the following properties
\begin{enumerate}
\item  $\theta_0=\frac{w_1}{cos(\beta\pi/2)}=\frac{2-\beta-\beta^2}{2cos(\beta\pi/2)}>0$,\\\\
\item  $\theta_1=\theta_{-1}=\frac{w_0+w_2}{2cos(\beta\pi/2)}=\frac{\beta(\beta+2)(\beta-1)}{8cos(\beta\pi/2)}<0$,\,\\\\
  for $|j|>1$, $\theta_j=\theta_{-j}=\frac{w_{j+1}}{2cos(\beta\pi/2)}
%=\frac{(-1)^j}{2cos(\beta\pi/2)}\left(-\frac{\beta}{2}\left(\begin{array}{c}\beta\\j+1\end{array}\right)+\frac{2-\beta}{2}
%\left(\begin{array}{c}\beta\\j \end{array}\right)\right)\leq0$,\\\\
=\frac{ \Gamma(j-\beta)}{2cos(\beta\pi/2)\Gamma(-\beta)\Gamma(j+1)}\left(1-\frac{\beta(\beta+1)}{2(j+1)}\right)  \leq0$,\\\\
\item  $\theta_{j+1}=\frac{(-\beta + j) (\beta + \beta^2 - 2 (2 + j))}{(2 + j) (\beta + \beta^2 - 2 (1 + j))}\theta_{j} $, for $|j|\geq2$,
 and $|\theta_{1} |>|\theta_{2} | \geq|\theta_{3} |\geq\cdots  $,\\\\
\item  $\sum\limits_{j=-\infty}^{\infty}\theta_j=0$, and $\sum\limits_{j=-M+i \atop j\neq0}^{i }|\theta_j|< \theta_0$.
\end{enumerate}
%and for the case $(p,q)=(1,-1)$ (WSGD scheme II) the weights $\theta_j$ satisfy
%\begin{enumerate}
%\item  $\sum\limits_{j=-\infty}^{\infty}\theta_j=0,\quad\sum\limits_{j=-M+i-1 }^{i+1}\theta_j> 0 ,\quad\theta_0=\frac{w_1^{(1,-1)}}{cos(\beta\pi/2)}=-\frac{2\beta+\beta^2}{4cos(\beta\pi/2)}>0$,\\\\
%\item  $\theta_1=\theta_{-1}=\frac{w_0^{(1,-1)}+w_2^{(1,-1)}}{2cos(\beta\pi/2)}=\frac{\beta^3+\beta^2-2\beta+8}{16cos(\beta\pi/2)}<0$,\\\\
%\item  $\theta_2=\theta_{-2}=\frac{w_3^{(1,-1)}}{2cos(\beta\pi/2)}=\frac{\beta(2-\beta)(\beta^2+\beta-8)}{12cos(\beta\pi/2)}\geq0$, and for $|j|>2$,\\\\
% $\theta_j=\theta_{-j}=\frac{w_{j+1}^{(1,-1)}}{2cos(\beta\pi/2)}
%=\frac{ \Gamma(j-\beta-1)}{2cos(\beta\pi/2)\Gamma(-\beta)\Gamma(j)}\left(1-\frac{(1 + \beta) (2 + \beta) (2 j - \beta)}{4 j (1 + j)}\right)\leq0$,\\\\
%%=\frac{(-1)^j}{2cos(\beta\pi/2)}\left(\frac{2+\beta}{4}\left(\begin{array}{c}\beta\\j\end{array}\right)+
%%\frac{2-\beta}{4}\left(\begin{array}{c}\beta\\j-2\end{array}\right)\right)\leq0$,\\\\
%\item  $\theta_{j+1}=\frac{ (1 + \beta + j) (\beta (1 + \beta) (2 + \beta) - 2 \beta (3 + \beta) j + 4 (1 + j)^2)}{(2 + j) (4 j^2 + \beta(1 + \beta) (2 + \beta) -
%   2 j \beta (3 + \beta))}\theta_j$.
%\end{enumerate}
In view of the approximations (\ref{sch22}), we know that the above discrete scheme of the Riesz fractional derivative also has second-order accuracy
\begin{equation}\label{sch23}
\frac{\partial^\beta u(x,t)}{\partial|x|^\beta }\big|_{x=x_k,t=t_j}=\frac{-1}{h^\beta}\sum_{j=-M+i}^{i}\theta_ju_{i-j}^k+\mathcal{O}(h^2).
\end{equation}
If the second-order WSGD discretization is substituted into the reaction-diffusion equations (\ref{2.1}),
we obtain the new semi-implicit WSGD finite difference scheme. And the same initial and boundary discretizations (\ref{2.9}) and (\ref{2.10}) are used. Then we can get the same formulations for the WSGD scheme as the ones for the fractional centered difference scheme (\ref{2.8})-(\ref{2.11}).
%Considering the same   initial conditions (\ref{2.9}) and   bounded conditions (\ref{2.10}), we can know that the   finite difference schemes  is

The semi-implicit fractional centered difference and the semi-implicit WSGD scheme have the same structure, and in particular, their discretized weights have the similar properties. So in the following, the two schemes have the same analyses, and we just focus on the first scheme.
%Since the properties of the discretized weights for the WSGD scheme II are different from the other two schemes, after performing the analyses for the first scheme we usually make a remark to illustrate the small differences between the proof of semi-implicit centered difference scheme and the semi-implicit WSGD scheme II.
%
%%\begin{remark}
%Compare the coefficients of the two discrete schemes of spatial Riesz fractional derivative, $g_i$ and $\theta_i$.
%And we can find that they have similar properties.
% Then below we just discuss the stability and convergence of the finite difference schemes (\ref{2.8})
% with initial conditions (\ref{2.9}) and  bounded conditions (\ref{2.10}).
%%\end{remark}
%
%---------------------------------------------------------------------------------------------------------------------------------------------------------------------

\section{Stability and convergence of the numerical schemes }
Now we first denote the two nonlinear terms  by
\begin{equation} \label{Nonl1}
f_1(N,P)=N\left(1-N-\frac{\varrho P}{P+N}\right)
\end{equation}
 and
\begin{equation} \label{Nonl2}
 f_2(N,P)=\sigma P\left(-\frac{\gamma+\kappa\delta P}{1+\kappa P}+\frac{N}{P+N}\right)
\end{equation}
 for convenience.
And assume  that they
%functions $f_1$ and $f_2$
satisfy  local Lipschitz condition, i.e.,
 there exists a positive constant $L$ such that
$$
|f_1(N_1,P_1)-f_1(N_2,P_2)|\leq L(|N_1-N_2|+|P_1-P_2|)
$$
and
$$
|f_2(N_1,P_1)-f_2(N_2,P_2)|\leq L(|N_1-N_2|+|P_1-P_2|),
$$
when $|N_1-N_2|\leq\epsilon_0$ and $|P_1-P_2|\leq\epsilon_0$ for a given positive constant $\epsilon_0$.
Then the discrete scheme (\ref{2.8}) can be rewritten as the following form
\begin{equation}\label{3.1}
\begin{aligned}
\left(\frac{1}{\mu}+\frac{D_1g_0}{h^\beta}\right)N_i^{k+1}=&\frac{1}{\mu}\sum_{n=1}^k(b_{k-n}-b_{k-n+1})N_i^n+\frac{1}{\mu}b_kN_i^0\\
  &-\frac{D_1}{h^\beta}\sum_{j=-M+i\atop j\neq0}^ig_jN_{i-j}^{k+1}+f_1(N_i^k,P_i^k),\\\\
\left(\frac{1}{\mu}+\frac{D_2g_0}{h^\beta}\right)P_i^{k+1}=&\frac{1}{\mu}\sum_{n=1}^k(b_{k-n}-b_{k-n+1})P_i^n+\frac{1}{\mu}b_kP_i^0\\ &-\frac{D_2}{h^\beta}\sum_{j=-M+i\atop j\neq0}^ig_jP_{i-j}^{k+1}+f_2(N_i^k,P_i^k),\\\\
\end{aligned}
\end{equation}
where $\mu=\Gamma(2-\alpha)\tau^{\alpha}$.
%We discuss the stability and convergence of the numerical schemes
%in two   discrete schemes, fractional centered difference scheme which based on the the fractional
%centered difference and Riemann-Liouville scheme which based on the discrete approximation for Riemann-Liouville derivatives.
%By compare the approximation scheme (\ref{sch23}) to (\ref{2.5}), we can know that the coefficient $g_j$ of fractional centered difference scheme
%and $\theta_j$ of Riemann-Liouville scheme have   similar properties.
%So we just discuss the stability and convergence of the  fractional centered difference   schemes.
 In this paper, we use $L_\infty$ norm of $\{u_i^k\}_{i=1}^{M-1}$ which is defined as
$$
\|u^k\|=\max_{1\leq i\leq M-1}|u^k_i|.
$$
Let $(\bar{N}_i^k,\bar{P}_i^k)$ be the approximate solution to the numerical scheme
and denote $\epsilon_i^k=N_i^k-\bar{N}_i^k$ and  $\varepsilon_i^k=P_i^k-\bar{P}_i^k$.
Then  under the small perturbations, there exists the following numerical stability result.

\begin{theorem}\label{Th3.1}%---------------------------------------------------------------------theorem----
The numerical schemes (\ref{2.8})-(\ref{2.10}) are stable and there
exist
\begin{equation}\label{3.2}
\begin{array}{llll}
\displaystyle
\| \epsilon^k \|\leq\frac{1}{(1-\alpha)-T^{\alpha}\Gamma(2-\alpha)L}(\| \epsilon^0
\|+\| \varepsilon^0
\|),\\\\
\displaystyle \| \varepsilon^k \|\leq\frac{1}{(1-\alpha)-T^{\alpha}\Gamma(2-\alpha)L}(\| \epsilon^0
\|+\| \varepsilon^0
\|) ,
\end{array}
\end{equation}
when $T < (1/(\Gamma(1-\alpha)L))^{1/\alpha} $.
\end{theorem}%-------------------------------------------------------theorem-----------

\begin{proof}
We prove this theorem by mathematical induction.
Using the first equation of (\ref{3.1}), we have
\begin{equation*}
\begin{aligned}
\left(\frac{1}{\mu}+\frac{D_1g_0}{h^\beta}\right)\epsilon_i^{k+1}=&\frac{1}{\mu}\sum_{n=1}^k(b_{k-n}-b_{k-n+1})\epsilon_i^n+\frac{1}{\mu}b_k\epsilon_i^0 \\\\
&-\frac{D_1}{h^\beta}\sum_{j=-M+i\atop j\neq0}^ig_j\epsilon_{i-j}^{k+1}+(f_1(N_i^k,P_i^k)-f_1(\bar{N}_i^k,\bar{P}_i^k)).
\end{aligned}
\end{equation*}
According to the numerical approximations (\ref{2.10}), when $i\neq 2,M-2$, the above equation can be rewritten as %$\epsilon_0^{k+1}$ and $\epsilon_M^{k+1}$
$$
\begin{aligned}
  \left(\frac{1}{\mu}+\frac{D_1g_0}{h^\beta}\right)\epsilon_i^{k+1}=&\frac{1}{\mu}\sum_{n=1}^k(b_{k-n}-b_{k-n+1})\epsilon_i^n+\frac{1}{\mu}b_k\epsilon_i^0
  -\frac{D_1}{h^\beta}\sum_{j=-M+i+1\atop j\neq0}^{i-1}g_j\epsilon_{i-j}^{k+1}\\
 &+\frac{D_1}{3h^\beta}(g_i\epsilon_2^{k+1}+g_{-M+i}\epsilon_{M-2}^{k+1}) -\frac{4D_1}{3h^\beta}(g_i\epsilon_1^{k+1}+g_{-M+i}\epsilon_{M-1}^{k+1})\\\\
 &+(f_1(N_i^k,P_i^k)-f_1(\bar{N}_i^k,\bar{P}_i^k))\\
 =&\frac{1}{\mu}\sum_{n=1}^k(b_{k-n}-b_{k-n+1})\epsilon_i^n+\frac{1}{\mu}b_k\epsilon_i^0
  -\frac{D_1}{h^\beta}\sum_{j=-M+i+1\atop j\neq0,i-2,-M+i+2}^{i-1}g_j\epsilon_{i-j}^{k+1}\\
 &+\frac{D_1}{ h^\beta}((g_i/3-g_{i-2})\epsilon_2^{k+1}+(g_{-M+i}/3-g_{-M+i+2})\epsilon_{M-2}^{k+1})\\\\
 & -\frac{4D_1}{3h^\beta}(g_i\epsilon_1^{k+1}+g_{-M+i}\epsilon_{M-1}^{k+1})+(f_1(N_i^k,P_i^k)-f_1(\bar{N}_i^k,\bar{P}_i^k)).
\end{aligned}
$$
Since $g_i/3-g_{i-2}>0$ and $g_{-M+i}/3-g_{-M+i+2}>0$,
we get
$$
\begin{aligned}
 &\left(\frac{1}{\mu}+\frac{D_1g_0}{h^\beta}\right)\|\epsilon_i^{k+1}\|  \\
 \leq&\frac{1}{\mu}\sum_{n=1}^k(b_{k-n}-b_{k-n+1})\|\epsilon_i^n\|+\frac{1}{\mu}b_k\|\epsilon_i^0 \|-\frac{D_1}{h^\beta}\sum_{j=-M+i+1\atop j\neq0,i-2,-M+i+2}^{i-1}g_j\|\epsilon_{i-j}^{k+1}\|\\
  &+\frac{D_1}{ h^\beta}((g_i/3-g_{i-2})\|\epsilon_2^{k+1}\|+(g_{-M+i}/3-g_{-M+i+2})\|\epsilon_{M-2}^{k+1}\|)\\\\
 &-\frac{4D_1}{3h^\beta}(g_i\|\epsilon_1^{k+1}\|+g_{-M+i}\|\epsilon_{M-1}^{k+1}\|)+L(\| N_i^k-\bar{N}_i^k \|+ \|P_i^k-\bar{P}_i^k \|).
\end{aligned}
$$
From the above inequality we know that
\begin{equation*}
\begin{aligned}
 \left(\frac{1}{\mu}+\frac{D_1g_0}{h^\beta}\right)\|\epsilon^{k+1}\|
 \leq&\frac{1}{\mu}\sum_{n=1}^k(b_{k-n}-b_{k-n+1})\|\epsilon ^n\|+\frac{1}{\mu}b_k\|\epsilon ^0\|+L (\|\epsilon^k\|+ \|\varepsilon ^k \|)\\
 &-\frac{D_1}{h^\beta}\sum_{j=-M+i \atop j\neq0}^{i }g_j\|\epsilon ^{k+1}\| .
\end{aligned}
\end{equation*}
%Since
%$$
%g_i/3-g_{i+1}=g_i\left(\frac{1}{3}-\left(1-\frac{\beta+1}{\beta/2+i+1}\right)\right)=\frac{-2ig_i}{3(\beta/2+i+1)}>0
%$$
%and
%$$
%\begin{aligned}
%\frac{g_{-M+i}}{3}-g_{-M+i-1}&=g_{-M+i-1}\left(\left(1-\frac{\beta+1}{\beta/2-M+i}\right)-\frac{1}{3}\right)\\
%&=\frac{(-2M+2i-1)g_i}{3(\beta/2+i+1)}>0,
%\end{aligned}
%$$
%we know $|g_i/3|<|g_{i+1}|$ and $ |g_{-M+i}/3|<|g_{-M+i-1}|$.
Together with  the fourth property of coefficient $g_i$, we can obtain the following inequality
 $$
\begin{aligned}
\left(\frac{1}{\mu}+\frac{D_1g_0}{h^\beta}\right)\|\epsilon^{k+1}\|  \leq&\frac{1}{\mu}\sum_{n=1}^k(b_{k-n}-b_{k-n+1})\|\epsilon ^n\|+\frac{1}{\mu}b_k\|\epsilon ^0\|+\frac{D_1g_0}{h^\beta} \|\epsilon ^{k+1}\|\\\\
&+L (\|\epsilon^k\|+ \|\varepsilon ^k \|).
 \end{aligned}
 $$
Therefore,
\begin{equation}\label{est}
  \|\epsilon ^{k+1}\|\leq \sum_{n=1}^k(b_{k-n}-b_{k-n+1})\|\epsilon ^n\|+  b_k\|\epsilon ^0\| +L\mu(\|\epsilon^k\|+ \|\varepsilon ^k \|).
\end{equation}
When $i=2$, we have
$$
\begin{aligned}
  &\left(\frac{1}{\mu}+\frac{D_1g_0}{h^\beta}\right)\epsilon_2^{k+1}-\frac{D_1}{ h^\beta} g_2/3\epsilon_2^{k+1}\\\\
 =&\frac{1}{\mu}\sum_{n=1}^k(b_{k-n}-b_{k-n+1})\epsilon_2^n+\frac{1}{\mu}b_k\epsilon_2^0
  -\frac{D_1}{h^\beta}\sum_{j=-M+3\atop j\neq0 ,-M+4}^{1}g_j\epsilon_{2-j}^{k+1}\\\\
 &+\frac{D_1}{ h^\beta}( g_{-M+2}/3-g_{-M+4})\epsilon_{M-2}^{k+1}\\\\  &-\frac{4D_1}{3h^\beta}(g_2\epsilon_1^{k+1}+g_{-M+2}\epsilon_{M-1}^{k+1})+(f_1(N_2^k,P_2^k)-f_1(\bar{N}_2^k,\bar{P}_2^k)).
\end{aligned}
$$
Then in the similar way as above, we can also get the estimate (\ref{est}) for $i=2$. If $i=M-2$, there exists a similar argument.

By almost the same deduction as $ \|\epsilon^{k+1}\|$, we can get
\begin{equation}
  \|\varepsilon^{k+1}\|\leq \sum_{n=1}^k(b_{k-n}-b_{k-n+1})\|\varepsilon ^n\|+ b_k\|\varepsilon^0\| +L\mu(\|\epsilon^k\|+ \|\varepsilon ^k \|).
\end{equation}
For $k=1$, it's easy to check that (\ref{3.2}) holds.
When $k>1$, suppose (\ref{3.2}) holds for $n=1,\cdots,k$.
Then  there exists
$$
\|\epsilon^{k+1}\|\leq\displaystyle\left(\frac{1-b_k+
2L\mu}{(1-\alpha)-T^{\alpha}\Gamma(2-\alpha)L} +b_k \right) (\| \epsilon^0
\|+\| \varepsilon^0
\|).
$$
We just need to prove
$$1-b_k+
2L\mu+b_k((1-\alpha)-T^{\alpha}\Gamma(2-\alpha)L)\leq1,$$
i.e.,
 $$2L\mu\cdot b_k^{-1}\leq \alpha+T^{\alpha}\Gamma(2-\alpha)L.$$
Using $\mu=\Gamma(2-\alpha)\tau^{\alpha}=\Gamma(2-\alpha)\frac{T^{\alpha}}{N^\alpha}$,    $b_k^{-1}<\frac{(k+1)^\alpha}{1-\alpha}$
and $\Gamma(1-\alpha)T^{\alpha}L< 1$,
we can obtain
$$
2L\mu\cdot b_k^{-1}<2L\cdot\Gamma(2-\alpha)\frac{T^{\alpha}}{N^\alpha}\cdot\frac{(k+1)^\alpha}{1-\alpha}\leq2L\cdot\Gamma(1-\alpha)T^{\alpha},
$$
and
$$
\alpha+T^{\alpha}\Gamma(2-\alpha)L>2L\cdot\Gamma(1-\alpha)T^{\alpha}.
$$
This implies that
$$1-b_k+
2L\mu+b_k((1-\alpha)-T^{\alpha}\Gamma(2-\alpha)L)\leq1.$$
So
$$
 \|\epsilon^{k+1}\|\leq \displaystyle
\frac{1}{(1-\alpha)-T^{\alpha}\Gamma(2-\alpha)L} (\| \epsilon^0
\|+\| \varepsilon^0
\|),
$$
and
$$
 \|\varepsilon^{k+1}\|\leq \displaystyle
\frac{1}{(1-\alpha)-T^{\alpha}\Gamma(2-\alpha)L} (\| \epsilon^0
\|+\| \varepsilon^0
\|).
$$
\end{proof}

\begin{theorem}%---------------------------------------------------------------------theorem----
Let $\{(N(x_i,t_k),P(x_i,t_k))\}_{i=1}^{M-1}$ and $\{(N_i^k,P_i^k)\}_{i=1}^{M-1}$ be the exact solutions of the
subdiffusive reaction-diffusion equation (\ref{2.1})-(\ref{2.3})  and of the
numerical schemes (\ref{2.8})-(\ref{2.10})  respectively and
 define
$\rho_i^k=N(x_i,t_k)-N_i^k$ and
$ \eta_i^k=P(x_i,t_k)-P_i^k$.
Then $\rho^k $ and $ \eta^k$ satisfy the following error estimates:
\begin{equation}\label{3.5}
\|\rho^{n}\|\leq C(\tau+h^2) \,and \,\|\eta^{n}\|\leq C(\tau+h^2),
\end{equation}
when $T < (1/(\Gamma(1-\alpha)L))^{1/\alpha} $.
\end{theorem}%---------------------------------------------------------------------theorem----

\begin{proof}
%As the nonlinear term is Lipschitz continuous, we get
%\[f_i(N_i^k,P_i^k)=f_i(N_i^{k+1},P_i^{k+1})+\mathcal{O}(\tau),\,\,i=1,2.\]
According to the first equation of (\ref{3.1}),  when $i\neq2,M-2$, we can get the following inequality
$$
\begin{aligned}
&\left(\frac{1}{\mu}+\frac{D_1g_0}{h^\beta}\right)|\rho_i^{k+1}|  \\\\
\leq&\frac{1}{\mu}\sum_{n=1}^k(b_{k-n}-b_{k-n+1})|\rho_i^n|+\frac{1}{\mu}b_k|\rho_i^0 |-\frac{D_1}{h^\beta}\sum_{j=-M+i+1\atop j\neq0,i-2,-M+i+2}^{i-1}g_j|\rho_{i-j}^{k+1}|\\\\
&+\frac{D_1}{ h^\beta}((g_i/3-g_{i-2})|\rho_2^{k+1}|+(g_{-M+i}/3-g_{-M+i+2})|\rho_{M-2}^{k+1}|)\\\\
&-\frac{4D_1}{3h^\beta}(g_i|\rho_1^{k+1}|+g_{-M+i}|\rho_{M-1}^{k+1}|)+L(| \rho_i^k |+ |\eta_i^k|)+r_{1i}^{k+1},
\end{aligned}
$$
where $r_1^{k+1}=O(\tau+h^2)$.
From the above inequality  we know that
\begin{equation*}
\begin{aligned}
 &\left(\frac{1}{\mu}+\frac{D_1g_0}{h^\beta}\right)\|\rho^{k+1}\|  \\\\
 \leq&\frac{1}{\mu}\sum_{n=1}^k(b_{k-n}-b_{k-n+1})\|\rho^n\|+\frac{1}{\mu}b_k\|\rho ^0\|-\frac{D_1}{h^\beta}\sum_{j=-M+i \atop j\neq0}^{i }g_j\|\rho^{k+1}\|\\\\
&  +L (\|\rho^k\|+ \|\eta ^k \|)+\|r_1^{k+1}\|.
\end{aligned}
\end{equation*}
In the view of  the fourth property of $g_i$, we obtain
$$
 \|\rho ^{k+1}\|\leq \sum_{n=1}^k(b_{k-n}-b_{k-n+1})\|\rho ^n\|+  b_k\|\epsilon ^0\| +L\mu(\|\rho^k\|+ \|\eta ^k \|)+\mu\|r_1^{k+1}\|,
$$
and
$$
 \|\eta^{k+1}\|\leq \sum_{n=1}^k(b_{k-n}-b_{k-n+1})\|\eta^n\|+  b_k\|\eta^0\| +L\mu(\|\rho^k\|+ \|\eta ^k \|)+\mu\|r_2^{k+1}\|,
$$
where $r_1^{k+1}=O(\tau+h^2)$ and $r_2^{k+1}=O(\tau+h^2)$. Similar to the discussions of the numerical stability for the special cases $i=2,M-2$, we know that the above inequalities hold for $i=1,2,\cdots,M-1$. Now suppose that
\begin{equation}\label{s2}
 \|\rho^{m}\| \le
 \left(\frac{b_{m-1}^{-1}}{(1-\alpha)-T^{\alpha}\Gamma(2-\alpha)L}\right)(\|\rho^0 \|+\|\eta^0 \|+\max_{1\leq n\leq m\atop i=1,2}\mu|| r_i^{n}||)
 \end{equation}
 and
\begin{equation}\label{s3}
\|\rho^{m}\| \le
 \left(\frac{b_{m-1}^{-1}}{(1-\alpha)-T^{\alpha}\Gamma(2-\alpha)L}\right)(\|\rho^0 \|+\|\eta^0 \|+\max_{1\leq n\leq m\atop i=1,2}\mu|| r_i^{n}||)
 \end{equation}
hold for $m=1,2,\cdots,k$. Next we prove that the estimates (\ref{s2}) and (\ref{s3}) hold for $m=k+1$.

When $m=1$, it's easy to check that (\ref{s2}) holds.
When $m=k+1$, together with $b_{k-n-1}^{-1}<b_n^{-1}$, we have
$$
\begin{aligned}
 ||\rho^{k+1}||\leq&
\sum_{n=0}^{k}\left(\frac{(b_n-b_{n+1})b_{k-n-1}^{-1}}{(1-\alpha)-T^{\alpha}\Gamma(2-\alpha)L}\right)(\|\rho^0 \|+\|\eta^0 \|+\max_{1\leq n\leq k\atop i=1,2}\mu|| r_i^{n}||)+b_n  ||\rho^0||
\\
\\
&+\frac{2 L\mu b_{n-1}^{-1}}{(1-\alpha)-T^{\alpha}\Gamma(2-\alpha)L}(\|\rho^0 \|+\|\eta^0 \|+\max_{1\leq n\leq k\atop i=1,2}\mu|| r_i^{n}||)+\mu|| r^{k+1}||
\\
\\
 \leq&
\left(\frac{b_{n}^{-1}}{(1-\alpha)-T^{\alpha}\Gamma(2-\alpha)L}\right)(\|\rho^0 \|+\|\eta^0 \|+\max_{1\leq n\leq k\atop i=1,2}\mu|| r_i^{n}||)+\mu|| r^{k+1}||.
\end{aligned}
$$
Notice that
$$
\rho^0_i=0, \,\eta^0_i=0,~~~~0\leq i\leq M.
$$
Together with
$b_n^{-1}\leq\frac{(n+1)^{\alpha}}{1-\alpha}$ and $(n+1)\tau\leq T$, we
obtain
%$$
%\begin{array}{lll}
%||\rho^{k+1}|| &\leq& \left(\displaystyle\frac{b_{n}^{-1}}{(1-\alpha)-T^{\alpha}\Gamma(2-\alpha)L}\right) \max\limits_{1\leq n\leq k+1\atop i=1,2}\mu|| r_i^{n}||)
%\\
%\\
%&\leq& \displaystyle\left(\frac{T^\alpha \Gamma(1-\alpha) }{(1-\alpha)-T^{\alpha}\Gamma(2-\alpha)L}\right)  \max\limits_{1\leq n\leq k+1\atop i=1,2}|| r_i^{n}||
%\\
%\\
%&\leq& \displaystyle C(\tau+h^2) .
%\end{array}
%$$
$$
||\rho^{k+1}||\leq C(\tau+h^2).
$$
In a similar way, we can get
$$
||\eta^{k+1}||\leq\displaystyle C(\tau+h^2).
$$

\end{proof}

%\subsection{RLD scheme}

%-------------------------------------------------------------------------------------------------------------------------------------------------------------------
\section{Positivity and boundedness of the analytical and numerical solutions of the space-time fractional predator-prey reaction-diffusion model}

%introduce maximum principle and  give

In this section, we prove that the analytical solutions of (\ref{2.1}) have positivity and boundedness. And then we demonstrate that the numerical solutions of the given schemes (\ref{2.8})-(\ref{2.10}) can preserve the properties: positivity and boundedness.

%\subsection{ Maximum principle for analytical solutions}
\subsection{Positivity and boundedness of the analytical solutions}
Firstly we introduce the maximum principle which is necessary for getting the positivity and boundedness of the analytical solutions. Here we consider the following one dimensional space-time fractional equation
 \begin{equation}\label{4.1}
  Lu=\frac{\partial^\alpha u}{\partial t^\alpha}-D_1\frac{\partial^\beta u}{\partial |x|^\beta}+c(x,t)u=f(x,t), ~~~~ (x,t) \in \Omega_T=[l,r] \times (0,T],
 \end{equation}
 where $\Omega_T$ is a bounded domain with Lipschitz continuous boundary.

\begin{theorem} \label{Th4.1}
Let the coefficient $c(x,t)\ge 0$ and the non-homogeneous term $f(x,t)\leq 0$ (resp. $f(x,t)\geq 0$) in $\Omega_T$. If $u\in C^{2,1}(\Omega_T)\bigcap C(\overline{\Omega}_T)$ is the solution of (\ref{4.1}), then the non-negative maximum (resp. non-positive minimum) of $u(x,t)$ in $\Omega_T$ (if exists) must reach at the parabolic boundary $\Gamma_T$, i.e.,
\begin{equation} \label{4.2}
\max_{\overline{\Omega}_T}u(x,t)\leq \max_{\Gamma{_T}} \lbrace u(x,t),0\rbrace\,\,\,\,\,\, ({\rm resp.}\,\,\min_{\overline{\Omega}_T}u(x,t)\geq \min_{\Gamma{_T}} \lbrace u(x,t),0\rbrace).
\end{equation}
\end{theorem}
In fact, if the non-negative maximum value of $u(x,t)$ is not at the boundary $\Gamma{_T}$, then there exists an interior point $(x^*,t^*)\in \Omega_T$ such that
\begin{equation}\label{me}
u(x^*,t^*)> \max_{\Gamma{ _T}}\lbrace u(x,t),0\rbrace~~{\rm and }~~ u(x^*,t^*) \ge \max_{\overline{\Omega}_T} u(x,t).
\end{equation}
%We  assume   $u$  is  smooth in the following discuss.
%Denote the minimum value of function $c(x,t)$ as $C_1$.
%Take a positive constant .  For any sufficient small $\varepsilon >0$,
We introduce  the auxiliary function $v(x,t)$ in the following form
\[v(x,t)=u(x,t)-\varepsilon  E_{\alpha,1}(bt^\alpha)\]
with $\varepsilon >0$ and $b>0$.
Choosing   sufficient small $\varepsilon $,  $v$ can be positive at the point $(x^*,t^*)$.
We know that, for any $(x,t) \in \Omega_T$, $v(x,t)$ satisfies
$$
\begin{aligned}
&\frac{\partial^{\alpha}v}{\partial t^{\alpha}}-D_1\frac{\partial^\beta v}{\partial|x|^\beta}+c(x,t)v\\\\
  =& f(x,t)-\varepsilon \big(b+D_1 C(x) +c(x,t)\big) E_{\alpha,1}(bt^\alpha),
\end{aligned}
$$
where $C(x)=\frac{( (r - x)^{-\beta} + (x-l)^{-\beta}) \sec(\pi \beta /2)(1 - \beta) }{ 2 \Gamma(
  2 - \beta)}$. Note that when $1<\beta<2$, $\sec(\pi \beta /2)( 1 - \beta) >0$. And because of $f(x,t)\leq0$, we deduce that
\begin{equation}\label{t12}
\frac{\partial^{\alpha}v}{\partial t^{\alpha}}-D_1\frac{\partial^\beta v}{\partial|x|^\beta}+c(x,t)v<0,
\end{equation}
at all interior points. While at the point $(x^*,t^*)$, we already have the result \cite{yu2013}
 \begin{equation}\label{t13}
\displaystyle\frac{\partial^{\alpha}v(x^*,t)}{\partial t^{\alpha}}|_{t=t^*}=\frac{\partial^{\alpha}u(x^*,t^*)}{\partial t^{\alpha}}
-\varepsilon b  E_{\alpha,1}(b(t^*)^\alpha)\geq0,
\end{equation}
%about Caputo derivative for at the point $x=x^*$
when $\varepsilon\leq \frac{(1-\alpha)T^{-\alpha}m^*}{\Gamma(2-\alpha)bE_{\alpha,1}(b(t^*)^\alpha)}$.
For the Riesz fractional derivative, we know that $u$ satisfies
$$
\begin{aligned}
\frac{\partial^\beta u}{\partial|x|^\beta}|_{x=x^*} &=\lim_{h\rightarrow0}\frac{1}{h^{\beta}}\big(-\sum_{j=[(l-r+x^*)/h]\atop j\neq0}^{[x^*/h]} g_ju(x^*-jh,t^*)-g_0u(x^*,t^*)\big)\\
&\leq\lim_{h\rightarrow0}\frac{1}{h^{\beta}}\big(\sum_{j=[(l-r+x^*)/h]\atop j\neq0}^{[x^*/h]} g_j(u(x^*,t^*)-u(x^*-jh,t^*)\big)<0.
\end{aligned}
$$
%Similarly, we can get the right Riemann-Liouville derivative of $v$ satisfy the inequality
%$$_xD_b^\alpha v|_{x=x^*}\leq \frac{v''(x^*,t)}{\Gamma(2-\beta)}+\frac{(b-x^*)^{1-\beta}}{\Gamma(1-\beta)}.$$
%According to (\ref{me}),  we get
%\[ u(x^*,t)-u(x^*-jh,t)<0.\]
%, it follows that
Then for sufficient small $\varepsilon$,
\begin{equation}\label{t14}
\frac{\partial^\beta v}{\partial|x|^\beta}|_{x=x^*}= \frac{\partial^\beta u(x,t^*)}{\partial|x|^\beta}|_{x=x^*} +\varepsilon  D_1 E_{\alpha,1}(b(t^*)^\alpha) C(x)|_{x=x^*}\leq 0.
\end{equation}
%According to
%$\frac{\partial^2 v}{\partial x^2}|_{x=x^*}\leq0$ and $\Gamma(1-\beta)\leq0$, it follows that
%\[_aD_x^\alpha v|_{x=x^*}+_xD_b^\alpha v|_{x=x^*}\leq0. \]
Together with (\ref{t13}) and (\ref{t14}), we obtain
\[ \frac{\partial^{\alpha}v}{\partial t^{\alpha}}-D_1\frac{\partial^\beta v}{\partial|x|^\beta}+c(x,t)v\geq 0 \,\,\,\,\,\, at ~~~ (x^*,t^*),\]
which is contradictory with (\ref{t12}).

Similar analysis can be done for the case $f(x,t) \geq 0$. So from the above analysis  we arrive at Theorem \ref{Th4.1}.

Next we introduce the definitions of the upper and lower solutions, from which we can get positivity and boundedness of the analytical solutions.

\begin{definition}%[the upper and lower solutions method]--------------------------
Assume $u_i \,(i=1,2)$ solve the following equations
\bigskip
\begin{align}
\frac{\partial^{\alpha}u_i}{\partial t^{\alpha}}-D_i\frac{\partial^{\beta}u_i}{\partial |x|^{\beta}}&= f_i(u_1,u_2),\ \,x\in\Omega,\,t\in(0,T],\nonumber\\
Bu_i= \frac{\partial u_i}{\partial x} &= g_i(x,t),\quad\ \,x\in\partial\Omega,\,t\in(0,T],\label{4.3} \\
u_i(x,0)& =\varphi_i(x),\qquad x\in\Omega,\nonumber
\end{align}
and the nonlinear functions $f_1(\cdot,\cdot)$ is quasi-monotone decreasing and $f_2(\cdot,\cdot)$ is quasi-monotone increasing.
If there exist two functions $\tilde{u}_i(x,t)$ and $\utilde{u}_i(x,t)$ which satisfy
%\bigskip
\begin{align}
&B \tilde{u}_i  -g_i(x,t)\geq 0\geq   B\utilde{u}_i -g_i(x,t),\ \,x\in\partial\Omega,\,t\in(0,T],\label{4.4.1}\\
&\tilde{u}_i(x,0)-\varphi_i(x)\geq 0 \geq\utilde{u}_i(x,0)-\varphi_i(x),\quad \ x\in\bar{\Omega}, \label{4.4.2}
\end{align}
 and
\bigskip
\begin{align}
\frac{\partial^{\alpha}\tilde{u}_1}{\partial t^{\alpha}}&-D_1\frac{\partial^{\beta}\tilde{u}_1}{\partial |x|^{\beta}}-f_1(\tilde{u}_1,\utilde{u}_2)\geq 0\geq \frac{\partial^{\alpha}\utilde{u}_1}{\partial t^{\alpha}}-D_1\frac{\partial^{\beta}\utilde{u}_1}{\partial |x|^{\beta}}-f_1(\utilde{u}_1,\tilde{u}_2),\label{4.4.3}\\
\frac{\partial^{\alpha}\tilde{u}_2}{\partial t^{\alpha}}&- D_2\frac{\partial^{\beta}\tilde{u}_2}{\partial |x|^{\beta}}-f_2(\tilde{u}_1,\tilde{u}_2)\geq 0\geq \frac{\partial^{\alpha}\utilde{u}_2}{\partial t^{\alpha}}- D_2\frac{\partial^{\beta}\utilde{u}_2}{\partial |x|^{\beta}}-f_2(\utilde{u}_1,\utilde{u}_2).\label{4.4.4}
\end{align}
Then  we call $U(x,t)=(\tilde{u}_1(x,t),\tilde{u}_2(x,t))$ and $V(x,t)=(\utilde{u}_1(x,t),\utilde{u}_2(x,t))$   upper  and lower solutions of the system (\ref{4.3}).
\end{definition}
\bigskip
%Because of the maximum principle, we can prove the upper and lower solution theorem. In the article, we skip the process.

\begin{theorem}\label{Th4.2}
Suppose $\{f_1,f_2\}$ is Lipschitz continuous and mixed quasi-monotonous. If the upper and lower solutions,
$U(x,t)$ and $V(x,t)$, satisfy the inequality $V(x,t)\leq U(x,t)$, then (\ref{4.3}) has a unique solution in $[V(x,t), U(x,t)]$.
\end{theorem}
\begin{proof}
Let us define the following iteration
  \begin{equation}\label{A1}
    \left\{
     \begin{aligned}
     &\frac{\partial^{\alpha}\bar{u}_1^{(k)}}{\partial t^{\alpha}}-D_1\frac{\partial^\beta \bar{u}_1^{(k)}}{\partial |x|^\beta}+L\cdot \bar{u}^{(k)}_1=L\cdot \bar{u}^{(k-1)}_1+f_1(\bar{u}_1^{(k-1)},\underbar{u}_2^{(k-1)}),  \\
   &\frac{\partial^{\alpha}\bar{u}_2^{(k)}}{\partial t^{\alpha}}-D_2\frac{\partial^\beta \bar{u}_2^{(k)}}{\partial |x|^\beta}+L\cdot
  \bar{u}^{(k)}_2=L\cdot \bar{u}^{(k-1)}_2+f_2(\bar{u}_1^{(k-1)},\bar{u}_2^{(k-1)}),  \\
   &\frac{\partial^{\alpha}\underbar{u}_1^{(k)}}{\partial t^{\alpha}}-D_1\frac{\partial^\beta \underbar{u}_1^{(k)}}{\partial |x|^\beta}+L\cdot
   \underbar{u}^{(k)}_1=L\cdot \underbar{u}^{(k-1)}_1+f_1(\underbar{u}_1^{(k-1)},\bar{u}_2^{(k-1)}),  \\
   &\frac{\partial^{\alpha}\underbar{u}_2^{(k)}}{\partial t^{\alpha}}-D_2\frac{\partial^\beta \underbar{u}_2^{(k)}}{\partial |x|^\beta}+L\cdot
   \underbar{u}^{(k)}_2=L\cdot \underbar{u}^{(k-1)}_2+f_2(\underbar{u}_1^{(k-1)},\underbar{u}_2^{(k-1)}),  \\
   & B\bar{u}_i^{(k)} |_{\partial\Omega\times(0,T]}=B \underbar{u}_i^{(k)} |_{\partial\Omega\times(0,T]}
   =g_i(x,t)|_{\partial\Omega\times(0,T]},\,\,(i=1,2) \\
   &\bar{u}_i^{(k)}(x,0)=\underbar{u}_i^{(k)}(x,0)=\varphi_i(x),\,\,(x\in\bar{\Omega},\,i=1,2),
     \end{aligned}
     \right.
    \end{equation}
    where $L$ is the maximum of Lipschiz constants of $f_1$ and $f_2$,
    and denote  the initial iteration function as
 \begin{align}
 (\bar{u}^{(0)}_1,\bar{u}^{(0)}_2)=(\tilde{u}_1,\tilde{u}_2),\nonumber\\
 (\underbar {u}^{(0)}_1,\underbar{u}^{(0)}_2)=(\utilde{u}_1,\utilde{u}_2),\nonumber
 \end{align}
 where $\utilde{u}_1 \leq  \tilde{u}_1$ and $\utilde{u}_2 \leq \tilde{u}_2$.
According to the maximum principle Theorem \ref{Th4.1} and using the iteration (\ref{A1}), similar to the analysis \cite{yu2013},  we can get
$$
 \utilde{u}_i\leq \underbar{u}^{(1)}_i\leq\cdots\leq\underbar{u}^{(k)}_i\leq\bar{u}^{(k)}_i\leq\cdots
  \leq\bar{u}^{(1)}_i\leq\tilde{u}_i, \,\,   (i=1,2).
$$
Note that $f_1$ is quasi-monotone decreasing and $f_2$ is quasi-monotone increasing, then
 \begin{align}
 &\lim_{k\rightarrow+\infty}\bar{u}^{(k)}_i=\bar{u}_i(x,t),\nonumber\\
 &\lim_{k\rightarrow+\infty}\underbar{u}^{(k)}_i=\underbar{u}_i(x,t),\,\,(i=1,2),\nonumber
 \end{align}
 which satisfy $\bar{u}_i\geq \underbar{u}_i$.
Next we will check that
 \[\bar{u}_i=\underbar{u}_i=u_i,(i=1,2).\]
Denote $w_1=\bar{u}_1-\underbar{u}_1(\ge0)$,  $w_2=\bar{u}_2-\underbar{u}_2(\ge0)$. We know $w_i\ge0,\,i=1,2$.
Then we get the following inequalities
   \begin{equation}
   \left\{
    \begin{aligned}
      &\frac{\partial^{\alpha}w_1}{\partial t^{\alpha}}  - D_1\frac{\partial^\beta w_1}{\partial |x|^\beta} -L\cdot(w_1+w_2)\leq0,\\
      &\frac{\partial^{\alpha}w_2}{\partial t^{\alpha}}  - D_2\frac{\partial^\beta w_2}{\partial |x|^\beta} -L\cdot( w_1+w_2)\leq0,\\
      &B w_1 =0,\quad  B w_2 =0,     \\
      & w_1(x,0)=0,\quad w_2(x,0)=0.
       \end{aligned}
      \right.
     \end{equation}
%For getting   $w \equiv0$ in $\overline{\Omega}_T$, we take a constant $b>L$, and  for sufficient small $\varepsilon$, we introduce a new
The below arguments show that $w_1=w_2=0$ in $\overline{\Omega}_T$. Suppose that there exist $(\hat{x}_i,\hat{t}_i) \in \Gamma_T$ ($i=1,2$) and   $w_1$ and $w_2$ can obtain their maximum values at $(\hat{x}_i,\hat{t}_i)$, ($i=1,2$), respectively.
%, and we denoted the two  maximum values as $w_1(\hat{x}_1,\hat{t}_1)$ and $w_2(\hat{x}_2,\hat{t}_2)$
We may assume that $w_1 (\hat{x}_1, \hat{t}_1)\le w_2(\hat{x}_2, \hat{t}_2)$. Since  $ w_i (x, 0)=0$ and $w_i (x,t)\geq0$, we can get  $t^*_i\, (<\hat{t}_i) $ such that
 $w_i(x,t)<\frac{1}{2}w_i(\hat{x}_i,\hat{t}_i)$ for any $(x,t)\in\overline{\Omega}\times(0,t^{*}_i)$.
Suppose the function $u_i$  solve the equations
  $$
  \left\{
    \begin{aligned}
    &\frac{ \partial u_i(t)}{\partial t}=-\frac{1}{2}\frac{w_i(\hat{x}_i,\hat{t}_i)}{t^{*}_i}\quad {\rm in}\,\,(0,t^{*}_i),
     \\
      & u_i(0)=\frac{1}{2}w_i(\hat{x}_i,\hat{t}_i),
     \\
      &
      u_i(t) =0, \, t \in [t^*_i,T],
   \end{aligned}
      \right.
  $$
where $i=1,2$, and let
 \begin{equation*}
    \bar{w}_i(x,t)=w_i(x,t)+u_i(t)\quad {\rm in} \,~ \Omega\times[0,T].
     \end{equation*}
Then
$$
\frac{\partial^{\alpha}\bar{w}_2}{\partial t^{\alpha}}- \frac{\partial^\beta \bar{w}_2}{\partial |x|^\beta}\leq L (w_1+w_2) + (1+D_2 C(x) )\frac{ \partial^{\alpha} u_2(t)}{\partial t^{\alpha}},
$$
where $C(x)\geq0$.
So at $(\hat{x}_2, \hat{t}_2)$,  $\bar{w}_2$ satisfies
  \begin{equation*}
  \begin{aligned}
     \frac{\partial^{\alpha}\bar{w}_2(\hat{x}_2, \hat{t}_2)}{\partial t^{\alpha}}- \frac{\partial^\beta \bar{w}_2(\hat{x}_2, \hat{t}_2)}{\partial |x|^\beta}
    \leq  2L w_2(\hat{x}_2, \hat{t}_2) +(1+D_2 C(x) )\frac{ \partial^{\alpha} u_2(\hat{t}_2)}{\partial t^{\alpha}} .
  \end{aligned}
     \end{equation*}
While    $\frac{\partial^{\alpha} \bar{w}_2 (\hat{x}_2,t)}{\partial t^{\alpha}}|_{t=\hat{t}_2}\geq0$,
  $- \frac{\partial^\beta \bar{w}_2 (x,\hat{t}_2)}{\partial |x|^\beta}|_{x=\hat{x}_2}\geq0$,
and since
 \begin{equation*}
 \frac{ \partial^{\alpha} u_2(t)}{\partial t^{\alpha}}|_{t=\hat{t}_2} < -\frac{w_2(\hat{x}_2,\hat{t}_2)}{2\Gamma(1-\alpha)}\hat{t}^{-\alpha}_2,
\end{equation*}
we get
\begin{equation*}
    2L w_2(\hat{x}_2,\hat{t}_2)+\frac{ \partial^{\alpha} u_2(\hat{t}_2)}{\partial t^{\alpha}}<0,
     \end{equation*}
for $T\leq(\frac{1}{4L\Gamma(1-\alpha)})^\frac{1}{\alpha}$.  This gives a contradiction, so $w_2\equiv 0$.
Now going back to $\bar{w}_1$, we know that
  \begin{equation*}
    \frac{\partial^{\alpha}\bar{w}_1}{\partial t^{\alpha}}- \frac{\partial^2 \bar{w}_1}{\partial x^2}\leq L (w_1+w_2)+(1+D_1 C(x))\frac{ \partial^{\alpha} u_1(t)}{\partial t^{\alpha}}= L w_1 +(1+D_1 C(x))\frac{ \partial^{\alpha} u_1(t)}{\partial t^{\alpha}} .
     \end{equation*}
     %can be gotten at $(\hat{x}_1,\hat{t}_1)$.
Then $w_1\equiv0$, when $T\leq(\frac{1}{2L\Gamma(1-\alpha)})^\frac{1}{\alpha}$.
Repeating the same process as done in \cite{yu2013}, we have $w_1=w_2$ in $ \overline{\Omega}_T$ for any $T$.
So  $\bar{u}_1=\underbar{u}_1$  and   $\bar{u}_2=\underbar{u}_2$.  Set
$$
u_i(x,t)=\bar{u}_i=\underbar{u}_i,\,\, i=1,2.
$$
Then $u(x,t)=(u_1,u_2)$ solves (\ref{4.3}). And the uniqueness of the solution can be similarly proved as \cite{yu2013}.
\end{proof}

From \cite{yu2013}, we know that the monotone properties of the nonlinear terms $f_1(\cdot,\cdot)$  and $f_2(\cdot,\cdot)$ of (\ref{2.1}) defined in (\ref{Nonl1}) and (\ref{Nonl2}), i.e., $f_1(\cdot,\cdot)$ is quasi-monotone decreasing, $f_2(\cdot,\cdot)$ is quasi-monotone increasing.
 %=(1,L_1)(L_1>1/\gamma)
 %Denote  $$\tilde{u}_1(x,t)=\left\{
%    \begin{aligned}
%    &1\quad {\rm in}\,\,[-l,2l]\times[0,T],\\
%     \\
%      & 0   \quad {\rm in} \,\,else,
%   \end{aligned}
%      \right.
%  $$
% and
%   $$\tilde{u}_2(x,t)=\left\{
%    \begin{aligned}
%    &L_1\quad  {\rm in}\,\,[-l,2l]\times[0,T],\\
%     \\
%      & 0    \,\,\, \quad {\rm in} \,\,else,
%   \end{aligned}
%      \right.
%  $$
%  where $L_1>1/\gamma$.
% Choose an appropriate  mollifier, and we will get  smooth functions  $\tilde{u'}_1(x,t),\tilde{u'}_2(x,t)$.
 %To certify that the equations (\ref{2.1})-(\ref{2.3}) have lower and upper solutions,
Now we take
 $$U(x,t)=(\tilde{u}_1(x,t),\tilde{u}_2(x,t))=(1,L_1),\,\,(L_1>1/\gamma)$$
  and
  $$V(x,t)=(\utilde{u}_1(x,t),\utilde{u}_2(x,t))=(0,0).$$
  By calculating, we know  $U(x,t)$ and $V(x,t)$ satisfy the inequalities  (\ref{4.4.3}) and (\ref{4.4.4}).
If we specify the initial   condition of (\ref{2.1}) such that the given initial $ N(x,0)\in[0,1]$ and  $P(x,0)\in[0, L_1]$ for any $x\in(l,r)$, then the initial conditions satisfy (\ref{4.4.2}).
Thus we know that $V(x,t)=(0,0)$ and $ U(x,t)=(1,L_1)$  are lower and upper solutions of (\ref{2.1})-(\ref{2.3}); and then from Theorem \ref{Th4.2}, that the analytical solutions of (\ref{2.1})-(\ref{2.3}) are positive and bounded is obtained.

\subsection{Positivity and boundedness of the numerical solutions}
In this subsection, we show that the numerical schemes (\ref{2.8}) can preserve the positivity and boundedness of the corresponding analytical solutions, i.e.,
for any $i$ and $k$,   $0<N_i^k\leq1$ and $0<P_i^k\leq L_1$. Let the initial conditions satisfy $0<N_i^0\leq1$ and $0<P_i^0\leq L_1$.
 According to mathematical induction, we need to show  that, if it holds for $N_i^{k}$ and $P_i^{k}$,
%$0<N_i^n\leq1$ and $0<P_i^n\leq L_1$ hold for  $1<n\leq k$,
then it also holds for $N_i^{k+1}$ and $P_i^{k+1}$. This can be done as follows.
%For the nonlinear terms, we obtain the estimates from \cite{yu2013},
From \cite{yu2013}, we have the following estimates
$$
-\varrho N_i^k\leq f_1(N_i^k,P_i^k)\leq1-N_i^ k,
$$
$$
-\sigma\delta P_i^k\leq f_2(N_i^k,P_i^k)\leq\sigma(1-\gamma P_i^k),
$$
with $\mu\leq1$,
%Since $b_n>b_{n+1}$ implies that
%The estimates
and
$$
(1-b_1)N_i^k\leq\sum_{n=0}^{k-1}(b_n-b_{n+1})N_i^{k-n}+b_kN_i^0\leq \frac{1-b_1}{2}N_i^k+\frac{1+b_1}{2},
$$
$$
(1-b_1)P_i^k\leq\sum_{n=0}^{k-1}(b_n-b_{n+1})P_i^{k-n}+b_kP_i^0\leq \frac{1-b_1}{2}P_i^k+\frac{1+b_1}{2}L_1.
$$
Together with the above estimates and the numerical scheme (\ref{2.8}), we obtain
%the following inequalities hold imply that
\begin{equation}\label{4.10}
0<N_{i}^{k+1}+\frac{D_1\mu}{h^\beta}\sum_{j=-M+i}^ig_jN_{i-j}^{k+1}\leq1
\end{equation}
for $\mu<\frac{1-b_1}{\varrho}$, and
\begin{equation}\label{4.11}
0<P_{i}^{k+1}+\frac{D_2\mu}{h^\beta}\sum_{j=-M+i}^ig_j P_{i-j}^{k+1}\leq L_1
\end{equation}
for $\mu<\frac{1-b_1}{\sigma \delta}$.
The inequalities above can  yield  $0<N_i^{k+1}\leq 1$ and $0<P_i^{k+1}\leq L_1$. In fact, if $0<N_i^{k+1}\leq 1$ does't hold, then there exists some $i$
such that  solution $N_i^{k+1}$ satisfies
$$
 N_i^{k+1}\leq0\quad or \quad N_i^{k+1}> 1.
$$

Case 1. If $N_i^{k+1}\leq0$, then we choose the minimum in $\{N_i^{k+1}\}_{i=1}^{M-1}$. Denote it as $N_s^{k+1}$  which is non-positive.
Along with the left inequality of (\ref{4.10}), there exists
$$
\begin{aligned}
  N_{s}^{k+1}&+\frac{D_1\mu}{h^\beta}\sum_{j=-M+s+1}^{s-1}g_jN_{s-j}^{k+1}+\frac{D_1\mu g_s}{3h^\beta}(4N_1^{k+1}- N_2^{k+1})\\\\
  &+\frac{D_1\mu g_{-M+s}}{3h^\beta}(4N_{M-1}^{k+1}-N_{M-2}^{k+1})>0.
\end{aligned}
$$
Then we have
\begin{equation}\label{4.12}
\begin{aligned}
  \left(1+\frac{D_1\mu g_0}{h^\beta}\right)N_{s}^{k+1}
 >&\frac{D_1\mu}{h^\beta}\sum_{j=-M+s+1\atop j\neq0}^{s-1}-g_jN_{s-j}^{k+1}-\frac{4D_1\mu g_s}{3h^\beta}N_1^{k+1}\\\\
 &-\frac{4D_1\mu g_{-M+s}}{3h^\beta}N_{M-1}^{k+1}+\frac{D_1\mu g_s}{3h^\beta}N_2^{k+1}+\frac{D_1\mu g_{-M+s}}{3h^\beta}N_{M-2}^{k+1}.
\end{aligned}
\end{equation}
When $s\neq 2,M-2$, we rewrite the inequality in the following form
$$
\begin{aligned}
 &\left(1+\frac{D_1\mu g_0}{h^\beta}\right)N_{s}^{k+1} \\\\
  >&\frac{D_1\mu}{h^\beta}\sum_{j=-M+s+1\atop j\neq0,s-2,-M+s+2}^{s-1}-g_jN_{s-j}^{k+1}-\frac{4D_1\mu g_s}{3h^\beta}N_1^{k+1}-\frac{4D_1\mu g_{-M+s}}{3h^\beta}N_{M-1}^{k+1}\\\\
  &+\frac{D_1\mu }{h^\beta}(\frac{ g_s}{3}-g_{s-2})N_2^{k+1}+\frac{D_1\mu }{h^\beta}(\frac{ g_{-M+s}}{3}-g_{-M+s+2})N_{M-2}^{k+1}.
\end{aligned}
$$
%Since $g_{j+1}=\left(1-\frac{\beta+1}{\beta/2+j+1}\right)g_j$, we get
%$$
%\begin{aligned}
%  g_{l-2}-g_l/3=&g_{l-2}\left(1-\frac{1}{3}\left(1-\frac{\beta+1}{\beta/2+l}\right)\left(1-\frac{\beta+1}{\beta/2+l-1} \right)\right)\\\\
%  =&g_{l-2}\frac{2(\beta^2/4+l^2)+4\beta l-2}{3(\beta/2+l)(\beta/2+l-1)}<0,
%\end{aligned}
%$$
%and
%$$
%\begin{aligned}
%  &g_{-M+l+2}- g_{-M+l}/3 \\\\
%  =&g_{-M+l}\left( \left(1-\frac{\beta+1}{\beta/2-M+l+1}\right)\left(1-\frac{\beta+1}{\beta/2-M+l+2} \right)-\frac{1}{3}\right)<0.
%\end{aligned}
%$$
Since $-g_s\leq -g_{s-2}$ and $-g_{-M+s}\leq -g_{-M+s+2}$, then $\frac{ g_s}{3}-g_{s-2}\geq 0$ and  $\frac{ g_{-M+s}}{3}-g_{-M+s+2}\geq 0$.
In view of $-g_j\geq0$ for any $j\neq 0$, we get
$$
\begin{aligned}
 &\left(1+\frac{D_1\mu g_0}{h^\beta}\right)N_{s}^{k+1}\\\\
  >&\frac{D_1\mu}{h^\beta}\sum_{j=-M+s+1\atop j\neq0,s-2,-M+s+2}^{s-1}-g_jN_{s}^{k+1}-\frac{4D_1\mu g_s}{3h^\beta}N_s^{k+1}-\frac{4D_1\mu g_{-M+s}}{3h^\beta}N_{s}^{k+1}\\\\
  &+\frac{D_1\mu }{h^\beta}(\frac{ g_s}{3}-g_{s-2})N_s^{k+1}+\frac{D_1\mu }{h^\beta}(\frac{ g_{-M+s}}{3}-g_{-M+s+2})N_{s}^{k+1}.
\end{aligned}
$$
%Then
%$$
%\left(1+\frac{\mu g_0}{h^\beta}\right)N_{l}^{n+1}>\frac{\mu}{h^\beta}\sum_{j=-M+l\atop j\neq0}^l-g_jN_{l}^{n+1}.
%$$
This implies
$$
\left(1+\frac{D_1\mu  }{h^\beta}\sum_{j=-M+s}^sg_j\right)N_{s}^{k+1}>0.
$$
Combining with the properties of $g_i$,   we know $1+\frac{D_1\mu}{h^\beta}\sum\limits_{j=-M+s}^sg_j>0$. So   $N_{s}^{k+1}>0$.
This contradicts the assumption.

If $s=2$, then along with the inequality (\ref{4.12}) we have
$$
\begin{aligned}
 &\left(1+\frac{D_1\mu g_0}{h^\beta}\right)N_{2}^{k+1}-\frac{D_1\mu g_2}{3h^\beta}N_2^{k+1}\\\\
 >&\frac{D_1\mu}{h^\beta}\sum_{j=-M+1\atop j\neq0}^{1}-g_jN_{2-j}^{k+1}-\frac{4D_1\mu g_2}{3h^\beta}N_1^{k+1}-\frac{4D_1\mu g_{-M+2}}{3h^\beta}N_{M-1}^{k+1}+\frac{D_1\mu g_{-M+2}}{3h^\beta}N_{M-2}^{k+1}.
\end{aligned}
$$
After   calculating, we get
$$
 \left(1+\frac{D_1\mu  }{h^\beta}\sum_{j=-M+2  }^{2 }-g_j\right)N_{2}^{k+1} >0  .
$$
Repeating the arguments above, we know that there exists a contradiction.
If $s=M-2$, the similar argument works.

Case 2. If $N_i^{k+1}>1$, then there exists a maximum  $N_s^{k+1}$ in $\{N_i^{k+1}\}_{i=1}^{M-1}$.
Since the right inequality of (\ref{4.10}) implies
$$
N_{s}^{k+1}+\frac{D_1\mu}{h^\beta}\sum_{j=-M+s}^sg_jN_{s-j}^{k+1}\leq1,
$$
%together with Lemma \ref{lemma2.1},
we have
$$
\left(1+\frac{D_1\mu}{h^\beta}\right)N_{s}^{k+1}\leq1+\frac{D_1\mu}{h^\beta}\sum_{j=-M+s\atop j\neq0}^s-g_jN_{s}^{k+1}.
$$
Rewriting the inequality, it follows that
$$
\left(1+\frac{D_1\mu}{h^\beta}\sum_{j=-M+s}^sg_j\right)N_{s}^{k+1}\leq1.
$$
Noting that $1+\frac{D_1\mu}{h^\beta}\sum\limits_{j=-M+s}^sg_j>1$,  $N_s^{k+1}$ needs to satisfy $N_s^{n+1}<1$, which is contradictory with the assumption.

%For $\frac{1}{\mu}=\frac{1}{\Gamma(2-\alpha)\tau^\alpha}\geq1$, we know $0<N_i^{k+1}\leq1$ and $0<P_i^{k+1}\leq L_1$.
The argument for $P_i^{k+1}$ is similar to the discussion of $N_i^{k+1}$.

%---------------------------------------------------------------------------------------------------------------------------------------------------------------
\section{Numerical experiments}

To verify the above theoretical results, we present some numerical results of the two numerical schemes (fractional centered difference scheme and WSGD scheme) with Neumann boundary.
%\begin{example}%---------------------
%For the fractional predator-prey model
%\begin{equation}\label{5.1}
%\begin{array}{llll}
%\displaystyle\frac{\partial^\alpha N}{\partial t^\alpha} &= &D_1 \displaystyle\frac{\partial^\beta N}{\partial |x|^\beta}+N\left(1-N-\frac{\varrho P}{P+N}\right),  \\
%\\
%\displaystyle\frac{\partial^\alpha P}{\partial t^\alpha} &= &D_2
%\displaystyle\frac{\partial^\beta P}{\partial |x|^\beta}+\sigma P
%\left(-\frac{\gamma+\kappa\delta   P}{1+\kappa P}+\frac{N}{P+N}\right),
%\end{array}
%\end{equation}
And the initial conditions are taken as \cite{ana2011}
\begin{align}
&N(x,0)=0.113585+0.0214cos(\pi x),\nonumber\\
&P(x,0)=0.471397+0.0066cos(\pi x).\nonumber
\end{align}
%where $(\bar{N},\bar{P})$ denote the equilibrium point.
%\end{example}
In this section, we fix the parameters $\sigma=1,\,\,\varrho=1.1,\, \,\gamma=0.05, \,\,\kappa=1$, and $\delta=0.5$,
%From the nonlinear terms
%$\bar{N}(1-\bar{N}-\frac{\varrho \bar{P}}{\bar{P}+\bar{N}})=0$ and
%$\sigma \bar{P}(-\frac{\gamma+\kappa\delta \bar{P}}{1+\kappa \bar{P}}+\frac{\bar{N}}{\bar{P}+\bar{N}})=0$,
%it is easy to obtain the equilibrium point $(\bar{N},\bar{P})=(0.113585,0.471397)$.
  use the diffusion coefficients $D_1=0.005$ and $D_2=0.2$ and take the domain $(0,1)\times(0,1)$.

In Figs \ref{fig1}-\ref{fig5}, we use the fractional centered difference scheme to show   different numerical solutions preserving the positivity and boundedness with the mesh $h=\tau=0.01$.
%The figures present the numerical solutions for different   $\alpha$ and $\beta$.
  We observe that the orders $\alpha$ and $\beta$ affect the shape of the solutions obviously.
 %When $\alpha=2$ and $\beta=1$, the system is  predator-prey model.
 % which is completely integrable.
%The Figs 1-5 present the numerical solutions for different order  $\alpha$ and $\beta$.
Fig. \ref{fig1} shows the solutions of  classical predator-prey reaction-diffusion model with $\alpha=1$ and $\beta=2$.
 When $\alpha$ tends to 1 and $\beta$ to 2, the numerical solutions of the fractional predator-prey reaction-diffusion equations are also convergent to the solutions of the classical ones.

\begin{figure}[htp]
\begin{center}
  % replace aims_logo.eps by your figure file name
  \includegraphics[width=5in]{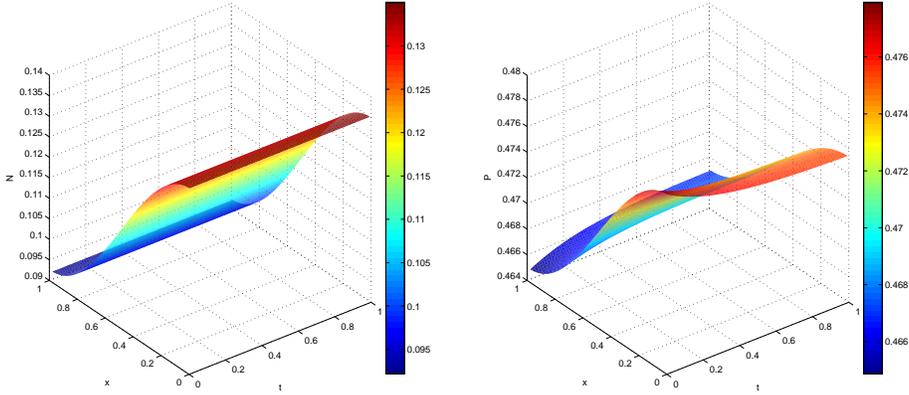}\\
  \caption{Numerical solution for the case $\alpha=1$ and $\beta=2$}\label{fig1}
  \end{center}
\end{figure}

\begin{figure}[htp]
\begin{center}
  \includegraphics[width=5in]{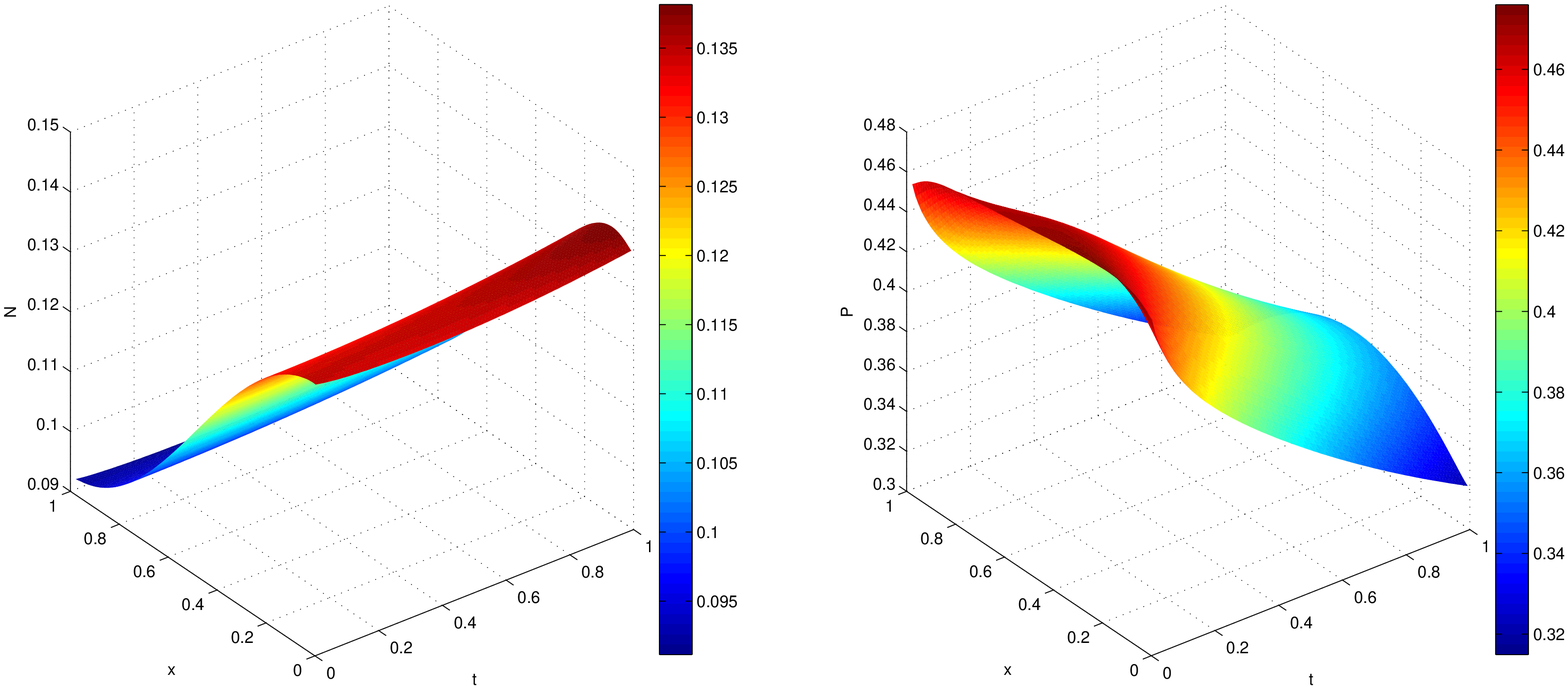}\\
  \caption{Numerical solution for the case $\alpha=1$ and $\beta=1.9$}\label{fig2}
  \end{center}
\end{figure}

\begin{figure}[htp]
\begin{center}
  \includegraphics[width=5in]{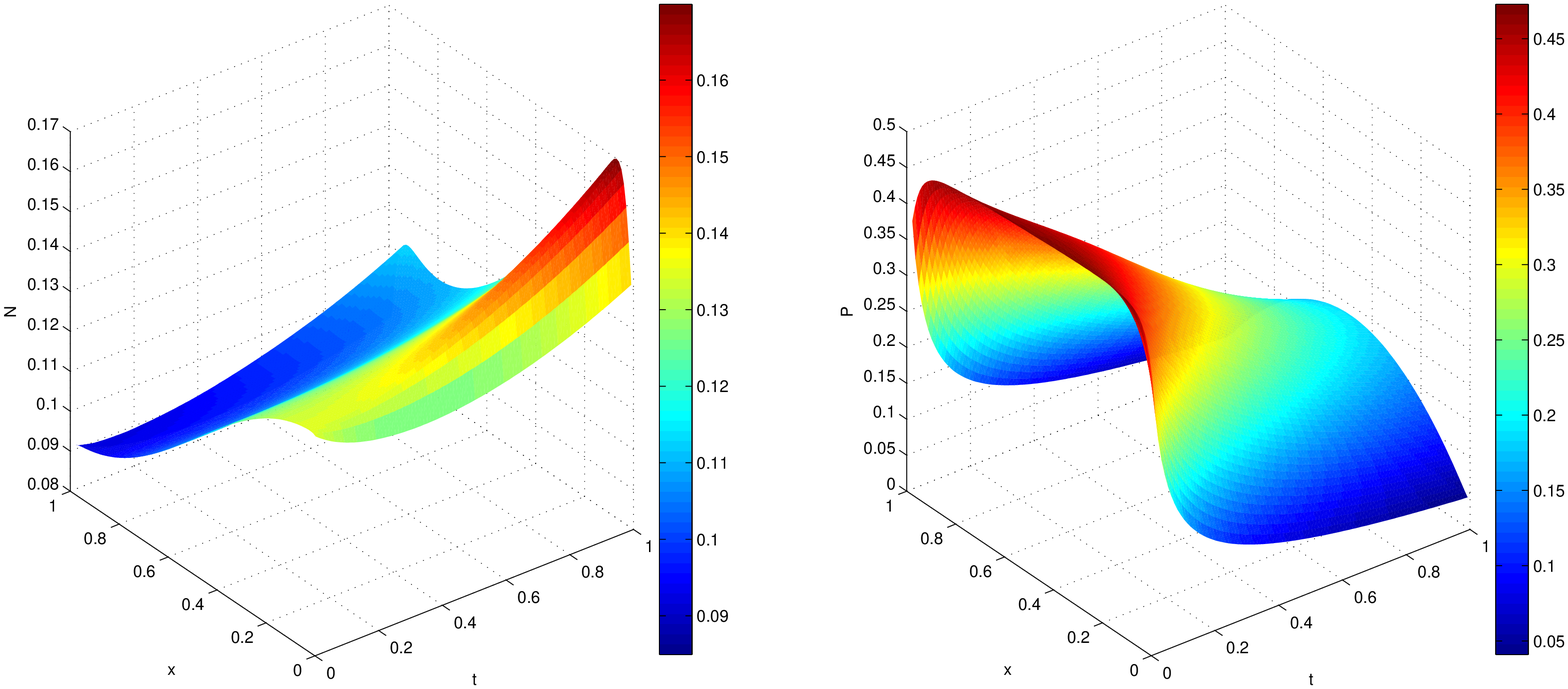}\\
  \caption{Numerical solution for the case $\alpha=1$ and $\beta=1.5$}\label{fig3}
  \end{center}
\end{figure}

\begin{figure}[htp]
\begin{center}
  \includegraphics[width=5in]{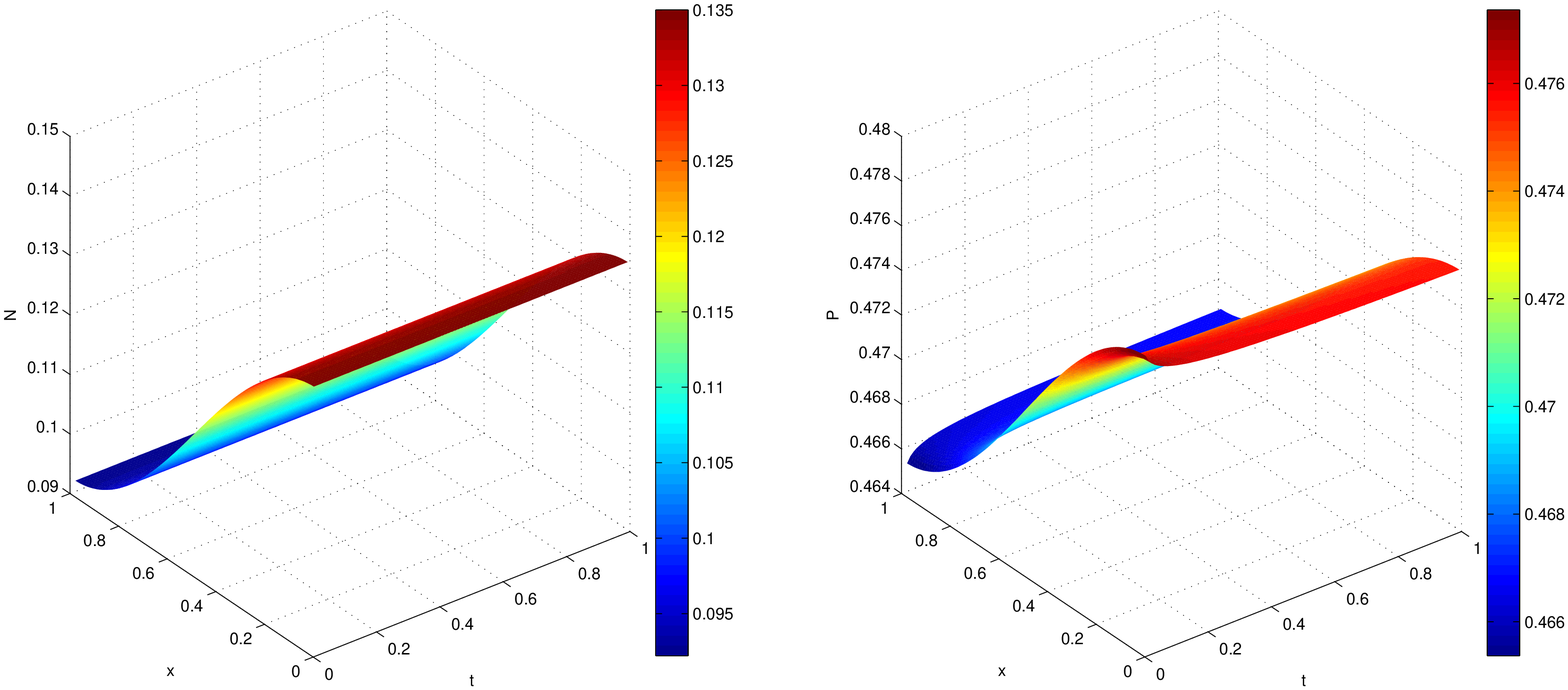}\\
  \caption{Numerical solution for the case $\alpha=0.5$ and $\beta=2$}\label{fig4}
  \end{center}
\end{figure}

\begin{figure}[htp]
\begin{center}
  \includegraphics[width=5in]{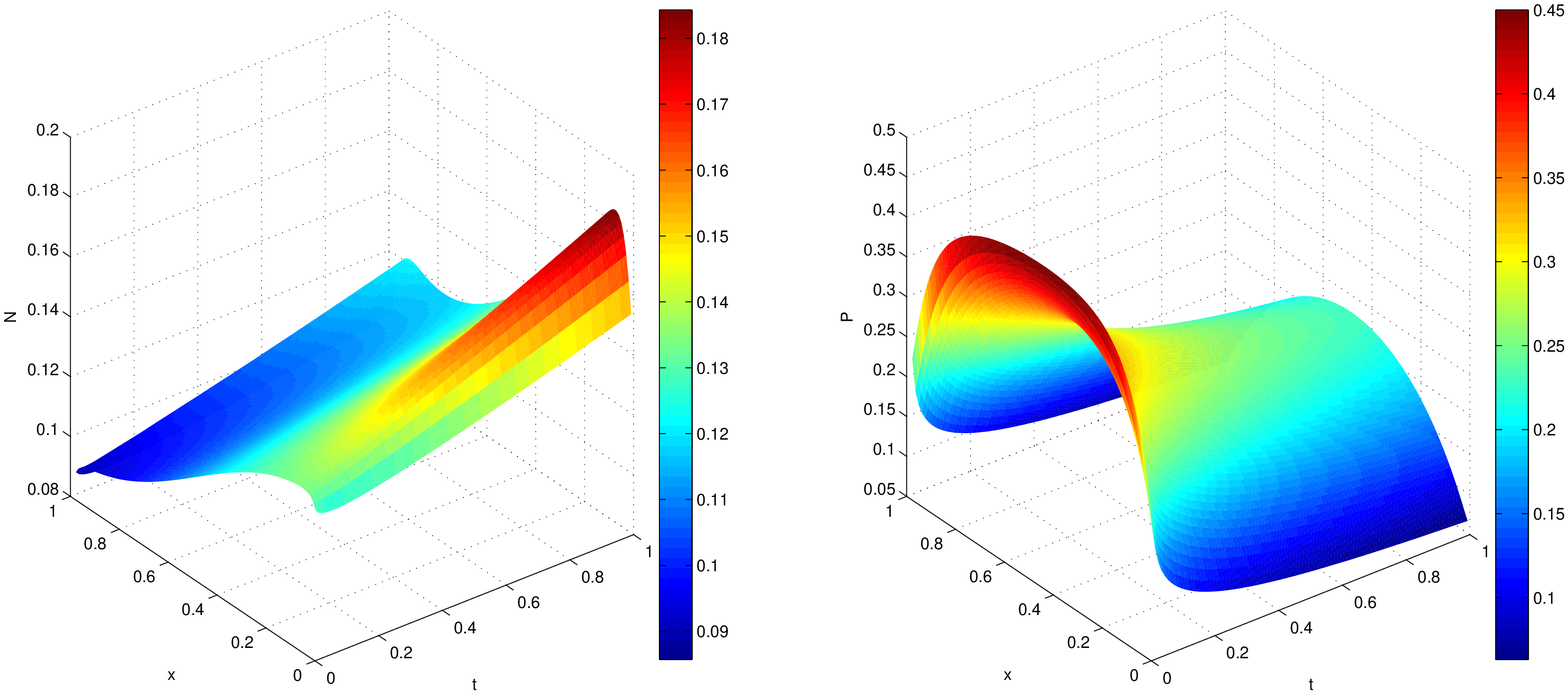}\\
  \caption{Numerical solution for the case $\alpha=0.5$ and $\beta=1.5$}\label{fig5}
  \end{center}
\end{figure}

To get the convergent order in time and space, we consider the nonhomogeneous equations:
\begin{equation*}
\begin{array}{llll}
\displaystyle\frac{\partial^\alpha N}{\partial t^\alpha} &= &D_1 \displaystyle\frac{\partial^\beta N}{\partial |x|^\beta}+N\left(1-N-\frac{\varrho P}{P+N}\right)+f(x,t),  \\
\\
\displaystyle\frac{\partial^\alpha P}{\partial t^\alpha} &= &D_2
\displaystyle\frac{\partial^\beta P}{\partial |x|^\beta}+\sigma P
\left(-\frac{\gamma+\kappa\delta   P}{1+\kappa P}+\frac{N}{P+N}\right)+g(x,t).
\end{array}
\end{equation*}
Suppose that the exact solution is
\[N(x,t)=P(x,t)=(t+1)^2x^2(1-x)^2,\]
and the nonhomogeneous terms are
\begin{equation*}
\begin{aligned}
  f(x,t)=&2(-1+x)^2x^2(\frac{t^{1-\alpha}}{\Gamma(2-\alpha)}  +\frac{t^{2-\alpha}}{\Gamma(3-\alpha)})\\\\
   &+\frac{D_1(1+t)^2}{\Gamma(5-\beta)}x^{-\beta}( \frac{(-1+x)^2x^\beta}{(1-x)^\beta}(12x^2-6x\beta+(-1+\beta)\beta)\\\\
  &+x^2(12(-1+x)^2+(-7+6x)\beta+\beta^2) )\sec(\frac{\pi\beta}{2})\\\\
& -(t+1)^2x^2(1-x)^2(1-(t+1)^2x^2(1-x)^2-\varrho/2),
  \end{aligned}
  \end{equation*}
 \begin{equation*}
\begin{aligned}
  g(x,t)=&2(-1+x)^2x^2(\frac{t^{1-\alpha}}{\Gamma(2-\alpha)}+\frac{t^{2-\alpha}}{\Gamma(3-\alpha)}) \\\\
 &+\frac{D_2(1+t)^2}{\Gamma(5-\beta)}x^{-\beta}( \frac{(-1+x)^2x^\beta}{(1-x)^\beta}(12x^2-6x\beta+(-1+\beta)\beta)\\\\
  & +x^2(12(-1+x)^2+(-7+6x)\beta+\beta^2) )\sec(\frac{\pi\beta}{2})\\\\
   &-\sigma(t+1)^2x^2(1-x)^2(-\frac{\gamma+\kappa\delta(t+1)^2x^2(1-x)^2}{1+\kappa(t+1)^2x^2(1-x)^2}+1/2).
  \end{aligned}
  \end{equation*}

%   \begin{equation}       %开始数学环境
%A=\frac{\mu D_i}{h^\beta}\left(                 %左括号
%  \begin{array}{cccccc}   %该矩阵一共3列，每一列都居中放置
%     g_0 +\frac{4}{3} g_{1}   & g_{-1} - \frac{1}{3} g_{1} & g_{-2} &  \cdots & g_{-M+3} -\frac{1}{3} g_{-M+1}  &g_{-M+2} +\frac{4}{3} g_{-M+1}  \\  %第一行元素
%    g_1  +\frac{4}{3} g_{2}    & g_0 -\frac{1}{3} g_{2}    &  g_{-1}  &        & g_{-M+4}-\frac{1}{3} g_{-M+2}  & g_{-M+3}+\frac{4}{3} g_{-M+2}   \\  %第二行元素
%    \vdots      &        &          &   \ddots     &         &   \\
%    g_{M-2} +\frac{4}{3} g_{M-1} &  g_{M-3}-\frac{1}{3} g_{M-1} & g_{M-4}&        & g_1 - \frac{1}{3} g_{-1}   & g_0 +\frac{4}{3} g_{-1}\\
%  \end{array}
%\right)                 %右括号
%\end{equation}
%we can get the following results.

Table \ref{table1} shows that the method has first-order accuracy for the time discretizations when $\alpha=0.5$ and $\beta=1.5$.
In the computations of the example, we take the spacial steplength $h=0.005$, which is small enough so that the error in space direction can be neglected for getting convergent rate.
% As   the exact solutions of the  equations can not be gotten, we take the solutions  at spacial steplength $h=0.01$ and temporal steplength $\tau=0.003125$  as the exact solution.
Table \ref{table2}-\ref{table4} show  that the fractional centered difference scheme has second-order accuracy for different cases, and Table \ref{table5} and \ref{table6} show  that the WSGD scheme also has second-order accuracy.
In these cases, we consider the errors with temporal steplength  $\tau=0.0001$, which is small enough so that the error in time direction can be neglected for getting the convergent rate.
%In the computations of the example, we take temporal  steplength $\tau=0.01$, which is small enough so that the  temporal error can be neglected for getting convergent rate in space direction.

\begin{table}[ht]
\caption{The numerical errors and convergent orders for the time discretization with $\alpha=0.5
$ and $\beta=1.5$ (the space derivative is discretized by the fractional centered difference scheme).}\label{table1}
 {\begin{tabular}{@{}ccccc@{}} \toprule
 $\tau (h=0.005)$ & $e_N(h,\tau)$ &rate & $e_P(h,\tau)$ &rate  \\
\midrule
$0.1$ & $  0.003082147971456  $ &   $   $ & $  0.003879257478438$ &   $   $  \\
$0.05 $&  $ 0.001660319792980 $ &   $0.89247  $  &  $ 0.002045107336033 $ &   $0.92360  $  \\
$0.025$ &  $ 8.702233156421824e-04$  &    $ 0.93200 $   &  $0.001057717860272$  &    $ 0.95122 $  \\
$0.0125$ &  $  4.486628201542942e-04 $  &   $    0.95575 $  &  $  5.405498347547111e-04$  &   $ 0.96846$ \\ \bottomrule
\end{tabular}}
\end{table}

%\begin{table} [ht]
%\tbl{The errors and   order    for   $\alpha=0.5
%$ and $\beta=1.5$.}
%{\begin{tabular}{@{}ccccc@{}} \toprule
% $h (\tau=0.0001)$ & $e_N(h,\tau)$ &order & $e_P(h,\tau)$ &order\\
% \colrule
%$0.1$ & $ 0.002459160159620 $ &   $   $ & $ 0.012293011238345  $ &   $   $  \\
%$0.05 $&  $ 4.388840438951450e-04$ &   $ 2.4863$  &  $ 0.001889498437314$ &   $  2.7018$  \\
%$0.025$ &  $  5.450419845520749e-05$  &    $3.0094 $  &  $   1.583898228462721e-04$  &    $  3.5765$ \\
%$0.0125$ &  $  2.187246690155131e-05$  &   $ 1.3173$ &  $  3.073456767939582e-05$  &   $  2.3655 $\\
% \botrule
%\end{tabular}}
%\end{table}
\begin{table} [ht]
\caption{The numerical errors and convergent orders of the fractional centered difference scheme with $\alpha=0.5
$ and $\beta=1.5.$}\label{table2}
{\begin{tabular}{@{}ccccc@{}} \toprule
 $h (\tau=0.0001)$ & $e_N(h,\tau)$ &order & $e_P(h,\tau)$ &order\\
 \midrule
$0.1$ & $ 0.002448567676975 $ &   $   $ & $ 0.012304879445640 $ &   $   $  \\
$0.05 $&  $ 4.333240674121927e-04$ &   $ 2.4984 $  &  $ 0.001897961774831$ &   $  2.6967 $  \\
$0.025$ &  $  4.078185842617405e-05$  &    $3.4094  $  &  $  1.770846542422655e-04 $  &    $  3.4219 $ \\
$0.0125$ &  $  7.936093047861137e-06$  &   $ 2.3614 $ &  $  2.735531025410687e-05$  &   $   2.6945 $\\
 \bottomrule
\end{tabular}}
\end{table}

\begin{table} [ht]
\caption{The numerical errors and convergent orders of the fractional centered difference scheme with $\alpha=0.5
$ and $\beta=1.1.$}\label{table3}
{\begin{tabular}{@{}ccccc@{}} \toprule
 $h (\tau=0.0001)$ & $e_N(h,\tau)$ &order & $e_P(h,\tau)$ &order\\
 \midrule
$0.1$ & $  0.001017521017230 $ &   $   $ & $ 0.004047202939302  $ &   $   $  \\
$0.05 $&  $ 1.098732300429083e-04$ &   $ 3.2174 $  &  $ 4.503780973269950e-04 $ &   $  3.1677  $  \\
$0.025$ &  $   1.518595837027292e-05 $  &    $ 2.8437 $  &  $   6.202433574753474e-05  $  &    $  2.8825 $ \\
%$0.0125$ &  $  5.452034409351644e-06$  &   $ 1.4779$ &  $  3.142088391842575e-05$  &   $  0.98111  $\\
 \bottomrule
\end{tabular}}
\end{table}

\begin{table} [ht]
\caption{The numerical errors and convergent orders of the fractional centered difference scheme with $\alpha=0.5
$ and $\beta=1.9.$}\label{table4}
{\begin{tabular}{@{}ccccc@{}} \toprule
 $h (\tau=0.0001)$ & $e_N(h,\tau)$ &order & $e_P(h,\tau)$ &order\\
 \midrule
$0.1$ & $  0.007646089314486 $ &   $   $ & $  0.044021358056344 $ &   $   $  \\
$0.05 $&  $ 0.002075492029463 $ &   $ 1.8813 $  &  $  0.010198249263744 $ &   $  2.1099  $  \\
$0.025$ &  $   4.102572803958426e-04 $  &    $ 2.3389 $  &  $  0.001878050620810 $  &    $ 2.4410  $ \\
$0.0125$ &  $  6.289181719891512e-05$  &   $ 2.7056 $ &  $  2.648747633180820e-04 $  &   $   2.8259 $\\
 \bottomrule
\end{tabular}}
\end{table}

\begin{table} [ht]
\caption{The numerical errors and convergent orders of the WSGD scheme with  $\alpha=0.5
$ and $\beta=1.1.$}\label{table5}
{\begin{tabular}{@{}ccccc@{}} \toprule
 $h (\tau=0.0001)$ & $e_N(h,\tau)$ &order & $e_P(h,\tau)$ &order\\
 \midrule
$0.1$ & $   8.143435340095817e-04 $ &   $   $ & $  0.003968570949633 $ &   $   $  \\
$0.05 $&  $  4.211307291067442e-04  $ &   $  0.95137  $  &  $  0.001501172252600 $ &   $  1.4025   $  \\
$0.025$ &  $    1.523958731684217e-04 $  &    $ 1.4664  $  &  $   6.213886912121139e-04 $  &    $ 1.2725 $ \\
$0.0125$ &  $  3.564037214530500e-05 $  &   $ 2.0962  $ &  $   2.003495389639826e-04$  &   $ 1.6330   $\\
 \bottomrule
\end{tabular}}
\end{table}

%
%\begin{table} [ht]
%\caption{The errors and   order    for   $\alpha=0.5
%$ and $\beta=1.5$ in RL scheme.}
%{\begin{tabular}{@{}ccccc@{}} \toprule
% $h (\tau=0.0001)$ & $e_N(h,\tau)$ &order & $e_P(h,\tau)$ &order\\
% \midrule
%$0.1$ & $  0.004271838621385    $ &   $   $ & $  0.018585869952296  $ &   $   $  \\
%$0.05 $&  $  0.000696921555996   $ &   $   2.6158  $  &  $  0.003771765462851  $ &   $   2.3009     $  \\
%$0.025$ &  $  0.000274754394254   $  &    $  1.3426    $  &  $  0.001686288868987  $  &    $    1.1614   $ \\
%$0.0125$ &  $   0.000135291017254  $  &   $ 1.0221  $ &  $  0.000774348070614   $  &   $    1.1228   $\\
% \bottomrule
%\end{tabular}}
%\end{table}

\begin{table} [ht]
\caption{The numerical errors and convergent orders of the WSGD scheme with $\alpha=0.5
$ and $\beta=1.9$ .}\label{table6}
{\begin{tabular}{@{}ccccc@{}} \toprule
 $h (\tau=0.0001)$ & $e_N(h,\tau)$ &order & $e_P(h,\tau)$ &order\\
 \midrule
$0.1$ & $ 0.017922317172642   $ &   $   $ & $  0.091921901959502  $ &   $   $  \\
$0.05 $&  $  0.005756704773643 $ &   $  1.6384   $  &  $   0.025183003437663 $ &   $  1.8680    $  \\
$0.025$ &  $  0.001233852922224   $  &    $  2.222    $  &  $  0.005115931266937  $  &    $  2.2994   $ \\
$0.0125$ &  $  0.000178819437800  $  &   $  2.7866   $ &  $    0.000661363526132  $  &   $     2.9515 $\\
 \bottomrule
\end{tabular}}
\end{table}

\section{Conclusion}
More than three decades ago, the classical diffusion is introduced into the predator-prey model, which leads to the predator-prey reaction-diffusion model. With the deep research on the anomalous diffusion, it is noticed that the anomalous diffusion is ubiquitous. It seems nature to introduce the anomalous diffusion to the predator-prey model, then we get the space-time fractional predator-prey reaction-diffusion model. This paper proves that the analytical solutions of the model are positive and bounded; and we provide the semi-implicit numerical schemes with the first-order accuracy in time and second-order accuracy in space. The proposed schemes are proven to be stable and preserve the positivity and boundedness. And the theoretical results are confirmed by numerical experiments.

\section*{Acknowledgements}
This work was supported by the National Natural Science Foundation of China under Grant No. 11271173.

%\section{Appendix}

\medskip
% The data information below will be filled by AIMS editorial staff
%Received xxxx 20xx; revised xxxx 20xx.

\medskip
\end{document}